\documentclass[11pt]{amsart}
\usepackage{amscd}      
\usepackage{amssymb}
\usepackage{amsmath, amsthm, graphics}
\usepackage[normalem]{ulem}
\usepackage{color}
\usepackage{xypic}  
\LaTeXdiagrams          
\usepackage[all,v2]{xy}
\xyoption{2cell} \UseAllTwocells \xyoption{frame} \CompileMatrices
\allowdisplaybreaks[4]

\usepackage[margin=1.05in]{geometry}

\usepackage{latexsym}
\usepackage{epsfig}
\usepackage{amsfonts}
\usepackage{enumerate}
\usepackage{times}

\newtheorem{prop}{Proposition}[section]

\theoremstyle{plain}
        \newtheorem{theorem}{Theorem}[section]
        \newtheorem{lemma}[theorem]{Lemma}
        \newtheorem{remark}[theorem]{Remark}
        \newtheorem{proposition}[theorem]{Proposition}

\theoremstyle{definition}
        \newtheorem{definition}[theorem]{Definition}

\theoremstyle{remark}

\theoremstyle{remark}

\numberwithin{equation}{section}

\newcommand{\Mbar}{\overline{\M}}

\newcommand{\X}{\mathcal{X}}
\newcommand{\Y}{\mathcal{Y}}

\newcommand{\M}{\mathcal{M}}
\newcommand{\C}{\mathcal{C}}

\newcommand{\B}{\mathcal{B}}

\newcommand{\sL}{\mathcal{L}}

\def\<{\left\langle}
\def\>{\right\rangle}

\newcommand{\complex}{{\mathbb C}}

\newcommand{\reals}{{\mathbb R}}

\newcommand{\integers}{{\mathbb Z}}

\newcommand{\calg}{{\mathcal G}}

\newcommand{\call}{{\mathcal L}}

\newcommand{\Ad}{\operatorname{Ad}}

\newcommand{\tr}{\operatorname{tr}}

\newcommand{\LL}{\vec{\mathfrak{l} } }

\title{A quantum Leray-Hirsch theorem for banded gerbes}
\author{Xiang Tang}
\address{(X. T.) Department of Mathematics, Washington University, St. Louis, MO 63130, USA}
\email{xtang@math.wustl.edu}
\author{Hsian-Hua Tseng}
\address{(H.-H. T.) Department of Mathematics, Ohio State University, Columbus, OH 43210, USA}
\email{hhtseng@math.ohio-state.edu}
\begin{document}
\date{\today}
\begin{abstract}
For a gerbe $\Y$ over a smooth proper Deligne-Mumford stack $\B$ banded by a finite group $G$, we prove a structure result on the Gromov-Witten theory of $\Y$, expressing Gromov-Witten invariants of $\Y$ in terms of Gromov-Witten invariants of $\B$ twisted by various flat $U(1)$-gerbes on $\B$. This can be viewed as a Leray-Hirsch type of result for Gromov-Witten theory of gerbes.
\end{abstract}

\maketitle 
\tableofcontents

\section{Introduction}
\subsection{Gerbes}
Gerbes arise naturally in the study of stacks. Every stack with non-trivial generic stabilizers is a gerbe over another stack, see e.g. \cite{bn}, \cite{acv}.

Inspired by the Physics paper \cite{hel-hen-pan-sh}, we studied thoroughly in \cite{ta-ts} the noncommutative geometry of \'etale gerbes. The purpose of this paper is to study an aspect of the quantum geometry of \'etale gerbes. Our main result is a complete determination of Gromov-Witten theory of a large class of \'etale gerbes called {\em banded} gerbes.

Let $\B$ be a smooth proper Deligne-Mumford stack. Let $G$ be a finite group. A gerbe $\Y\to \B$ is banded by $G$ if the associated $Out(G)$-torsor over $\B$ admits a section. It is well-known, see e.g. \cite{gi}, \cite{lurie}, \cite{br}, that the isomorphism class of such a gerbe of $\Y\to \B$ is classified by the cohomology group $H^2(\B, Z(G))$, where $Z(G)$ is the center of $G$. See Section \ref{subsec:structure-gerbe} for a review. In this paper, such a $\Y$ is called a {\em banded} $G$-gerbe over $\B$. 

\subsection{Gromov-Witten theory}
The Gromov-Witten invariants of $\Y$ are intersection numbers on moduli spaces of stable maps to $\Y$, see \cite{agv1, agv2}, \cite{cr2}, \cite{tseng} for introductions. Let $\beta\in H_2(\Y, \mathbb{Q})$ be a curve class. Let $$\Mbar_{g,n}(\Y, \beta)$$ be the moduli space of $n$-pointed, genus $g$, degree $\beta$ stable maps to $\Y$. The evaluation maps $$ev_i: \Mbar_{g,n}(\Y, \beta)\to I\Y$$ take values in the inertia orbifold $I\Y$ of $\Y$ with $I\Y:=\Y\times_{\Y\times \Y} \Y$. Let $\psi_i\in H^2(\Mbar_{g,n}(\Y,\beta), \mathbb{Q}), i=1,...,n$ be descendant classes. For $\phi_1,...,\phi_n\in H^\bullet(I\Y)$ and $a_1,...,a_n\in \mathbb{Z}_{\geq 0}$, one can associate the Gromov-Witten invariant 
$$\<\prod_{i=1}^n \phi_i\psi_i^{a_i}\>_{g,n,\beta}^\Y:=\int_{[\Mbar_{g,n}(\Y,\beta)]^{vir}}\prod_{i=1}^n ev_i^*\phi_i \psi_i^{a_i},$$
where $[\Mbar_{g,n}(\Y, \beta)]^{vir}$ is the {\em weighted} virtual fundamental class used in \cite{agv1} and \cite{tseng}.

Let $\widehat{G}$ be the set of isomorphism classes of irreducible unitary representations of $G$. For each $\rho\in \widehat{G}$, a locally constant $U(1)$-gerbe $c_\rho$ on $\B$ is constructed in \cite{ta-ts}. The class of $c_\rho$ in $H^2(\B, U(1))$ is determined by the class of $\Y\to \B$ via the map $H^2(\B, Z(G))\to H^2(\B, U(1))$ where $Z(G)\to U(1)$ is given by the restriction of $\rho$ to $Z(G)$.

By \cite{pry}, \cite{lu}, the holonomy of the gerbe $c_\rho$ defines a line bundle $\sL_{c_\rho}\to I\B$. By \cite{pry}, there is a canonical trivialization $$\theta: \otimes_{i=1}^nev_i^*\sL_{c_\rho}\to \underline{\mathbb{C}}$$ of the line bundle on $\Mbar_{g,n}(\B, \beta)$ for a given $\beta\in H_2(\B, \mathbb{Q})$.

Let $H^\bullet(I\B, \sL_{c_\rho})$ be the cohomology of $I\B$ with coefficients in $\sL_{c_\rho}$. For $\varphi_1,...,\varphi_n\in H^\bullet(I\B, \sL_{c_\rho})$ and $a_1,..., a_n\in \mathbb{Z}_{\geq 0}$, one can associate the $c_\rho$-twisted Gromov-Witten invariant 
$$\<\prod_{i=1}^n \varphi_i\psi_i^{a_i}\>_{g,n,\beta}^{\B, c_\rho}:=\int_{[\Mbar_{g,n}(\B,\beta)]^{vir}}\theta_*\left( \prod_{i=1}^nev_i^*\varphi_i\right) \prod_{i=1}^n \psi_i^{a_i}.$$

\subsection{Results}
In \cite{ta-ts}, an additive isomorphism 
\begin{equation}
I: H^\bullet(I\Y)\to \oplus_{\rho\in \widehat{G}}H^\bullet(I\B, \sL_{c_\rho})
\end{equation}
was obtained. For $\delta\in H^\bullet(I\Y)$, $I(\delta)$ decomposes accordingly:
$$I(\delta)=\sum_{\rho\in \widehat{G}} I(\delta)_\rho, \quad I(\delta)_\rho\in H^\bullet(I\B, \sL_{c_\rho}).$$
The following is the main result of this paper.

\begin{theorem}\label{thm:main}
Let $g\geq 0$, $n>0$, and $a_1,...,a_n\geq 0$ be integers. Let $\beta\in H_2(\B, \mathbb{Q})=H_2(\Y, \mathbb{Q})$ be a curve class. Let $\delta_1,..., \delta_n\in H^\bullet_{CR}(\Y):=H^\bullet(I\Y)$. Then   the following equality of GW invariants holds:
\begin{equation}\label{eqn:decomp_GW_inv}
\<\prod_{j=1}^n \delta_j\psi_j^{a_j}\>_{g,n,\beta}^\Y=\sum_{\rho\in \widehat{G}}\left(\frac{\text{dim}\, V_\rho}{|G|}\right)^{2-2g}\<\prod_{j=1}^n I(\delta_j)_\rho\psi_j^{a_j}\>_{g,n,\beta}^{\B, c_\rho}.
\end{equation}

\end{theorem}
The following comments are in order.

\begin{enumerate}

\item
Theorem \ref{thm:main} confirms the decomposition conjecture/gerbe duality stated in \cite[Conjecture 1.8]{ta-ts}. The decomposition conjecture was studied previously for certain classes of gerbes in \cite{ajt1, ajt2, ajt2.5}, \cite{johnson}. Theorem \ref{thm:main} recovers and extends all of these works. 

\item
A $G$-gerbe $\Y\to \B$ may be interpreted as a fiber bundle over $\B$ with fibers $BG$. In this point of view, Theorem \ref{thm:main} can be understood as a ``quantum Leray-Hirsch'' result, as it determines the Gromov-Witten theory of $\Y$ in terms of Gromov-Witten theory of $\B$ and information about the fiber $BG$. 

\item
The virtual pushforward formula, Lemma \ref{lem:pushforward} below, implies that the Gromov-Witten theory of $\B$ is determined by the Gromov-Witten theory of $\Y$. If $\B$ is a scheme, then \cite[Corollary 6.2]{pry} shows that $c_\rho$-twisted Gromov-Witten invariants of $\B$ coincide with Gromov-Witten invariants of $\B$ up to certain holonomy factors. So in this case Theorem \ref{thm:main} implies that Gromov-Witten theories of $\Y$ and $\B$ are equivalent. 

\end{enumerate}

\subsection{Idea of proof}
The main technical aspect of the proof of Theorem \ref{thm:main} concerns properties of the map $$\pi: \Mbar_{g,\beta}(\Y, d)\to \Mbar_{g,n}(\B,\beta)$$ obtained by composing a stable map to $\Y$ with $\Y\to \B$. Choosing orbifold structures at marked points selects components of $\Mbar_{g,n}(\Y, \beta)$ and $\Mbar_{g,n}(\B, \beta)$. In Section \ref{sec:degree}, the degree of $\pi$ over each component are explicitly evaluated by a detailed analysis of liftings of stable maps from target $\B$ to target $\Y$. The pushforward $\pi_*[\Mbar_{g,n}(\Y, \beta)]^{vir}$ is computed using the information about degrees. These form the main ingredients of the proof of Theorem \ref{thm:main} for abelian groups $G$. This is explained in Theorem \ref{thm:main-abelian}. 

Theorem \ref{thm:main} for general $G$ is obtained in two stages. The short exact sequence $$1\to Z(G)\to G\to G/Z(G)=: K\to 1$$ implies that the gerbe $\Y\to \B$ can be factored as $$\Y\to \Y'\to \B,$$ where $\Y\to \Y'$ is a banded $Z(G)$-gerbe and $\Y'\to \B$ is a banded $K$-gerbe. Theorem \ref{thm:main-abelian} applies to $\Y\to \Y'$ and expresses Gromov-Witten invariants of $\Y$ in terms of Gromov-Witten invariants of $\Y'$ with a certain twist $c'$. The $K$-gerbe $\Y'\to\B$ is necessarily trivial, i.e. $\Y'\simeq \B\times BK$. The twisted Gromov-Witten theory of such a product can be explicitly solved in terms of twisted Gromov-Witten theory of $\B$, see Theorem \ref{thm:product}. Theorem \ref{thm:main} is derived by combining the results for $\Y\to \Y'$ and $\Y'\to \B$. 

\subsection{Plan of the paper}
The proof of Theorem \ref{thm:main} is given in Section \ref{sec:decomp_thm}. The abelian case, Theorem \ref{thm:main-abelian}, is proven in Section \ref{sec:thm-abelian}. Theorem \ref{thm:product}, which concerns the twisted theory of a product, is proven in Section \ref{sec:pf_thm_prod}. Section \ref{sec:degree} is devoted to the detailed study of liftings of stable maps to a gerbe, and a proof of the required pushforward formula for virtual fundamental classes. Appendix \ref{app:vir_push} is concerned with virtual pushforward property. Appendix \ref{app:counting_result} derives a counting formula needed for the degree computation. Appendix \ref{app:twisted_bgamma} discusses gerbe-twisted Gromov-Witten theory of a classifying stack.

\subsection{Notations and conventions}

For a Deligne-Mumford $\mathbb{C}$-stack $\X$, its Chen-Ruan orbifold cohomology is defined as $$H_{CR}^\bullet(\X):=H^\bullet(I\X).$$ If $c$ is a flat $U(1)$-gerbe on $\X$, the holonomy of $c$ defines a line bundle $\sL_c$ on $I\X$. The $c$-twisted Chen-Ruan orbifold cohomology of $\X$ is defined as $$H_{CR}^\bullet(\X, c):=H^\bullet(I\X, \sL_c).$$

Let $\B$ be a proper DM stack over $\mathbb{C}$. Let $\pi: \Y\to \B$ be a banded $G$-gerbe. Let $(\widehat{\Y}, c)$ be the dual pair considered in \cite{ta-ts}. Since the $G$-gerbe $\Y\to \B$ is banded, the following description holds: 
\begin{equation}\label{eqn:decomp_dual}
\widehat{\Y}=\coprod_{\rho\in \widehat{G}}\widehat{\Y}_\rho, \quad \widehat{\Y}_\rho=\B.
\end{equation}

In this paper, the Gromov-Witten theory of a stack $\X$ is defined with insertions coming from the cohomology of the inertia orbifold $I\X$, as opposed to the rigidified inertia orbifold $\bar{I}\X$. This version of Gromov-Witten theory is described in \cite{agv1} and \cite{tseng}, and uses weighted virtual fundamental classes on moduli spaces of stable maps. 

In this paper, a smooth Deligne-Mumford $\mathbb{C}$-stack and its corresponding complex projective orbifold are treated as synonymous. Accordingly, all the orbifolds are equipped with canonical symplectic structures, coming from the K\"ahler structures, so that the results in \cite{ta-ts} can be applied.

\subsection{Acknowledgment}
We thank Professor Yongbin Ruan for his interests in this work and his continuous encouragement over the years. X. T. is supported in part by NSF grant DMS-1363250. H.-H. T. is supported in part by Simons Foundation Collaboration Grant and NSF grant DMS-1506551.   

\section{Decomposition of GW-invariants}\label{sec:decomp_thm}




\subsection{Structure of a banded $G$-gerbe}\label{subsec:structure-gerbe}

Following \cite{mo-pr-ind} and \cite{ta-ts}, we represent the orbifold $\B$ by a proper \'etale groupoid $\mathfrak{Q}\rightrightarrows M$ whose nerve spaces $\mathfrak{Q}^{(\bullet)}$ are disjoint unions of contractible open subsets of $\reals^n$.  A $G$-gerbe over $\B$ can be represented by an extension 
\[
G\times M\rightrightarrows M \stackrel{i}{\longrightarrow} \mathfrak{H}\rightrightarrows M \stackrel{j}{\longrightarrow} \mathfrak{Q}\rightrightarrows M,
\]
of the groupoid $\mathfrak{Q}\rightrightarrows M$ by the groupoid of a trivial bundle of groups $G\times M\rightrightarrows M$. 

Following \cite[Sec. 4.2]{ta-ts}, by a smooth section 
\[
\sigma: \mathfrak{Q}\longrightarrow \mathfrak{H},
\]
such that $j\circ \sigma=id$, we can write the groupoid $\mathfrak{H}$ as $G\rtimes _{\sigma, \tau}\mathfrak{Q}$, where $\tau$ is a nonabelian $2$-cocycle on $\mathfrak{Q}$ with value in $G$. More concretely, $\tau$ satisfies
\[
\tau(q_1, q_2)\tau(q_1q_2, q_3)=\Ad_{\sigma(q_1)}\big(\tau(q_2, q_3)\big)\tau(q_1, q_2q_3),\ \forall (q_1, q_2, q_3)\in \mathfrak{Q}^{(3)}. 
\]
As every component of $\mathfrak{Q}$ is contractible, $\mathfrak{H}$ is diffeomorphic to $G\times \mathfrak{Q}$ as a topological space. We choose the obvious splitting map $\sigma: \mathfrak{Q}\to \mathfrak{H}=G\times \mathfrak{Q}$ by mapping $q$ to $(e, q)$, where $e$ is the identity of $G$.  

As $G$ is a normal subgroup(oid) of $\mathfrak{H}$, the conjugation of $(e,q)$ on $G$ gives an automorphism of $G$. Following the discussion and notations in \cite[Sec. 3 and 4]{ta-ts}, we can describe the groupoid structure on $\mathfrak{H}$ as 
\[
(g_1, q_1)(g_2, q_2)=(g_1\operatorname{Ad}_{\sigma(q_1)}(g_2)\tau(q_1, q_2), q_1q_2).
\]
Recall that 
\begin{equation}\label{eq:tau}
(e,q_1)(e, q_2)=(\tau(q_1, q_2), e)(e, q_1q_2),
\end{equation}
and 
\begin{equation}\label{eq:tau-cocycle}
\tau(q_1, q_2)\tau(q_1q_2, q_3)=\operatorname{Ad}_{(e, q_1)}(\tau(q_2, q_3))\tau(q_1, q_2q_3). 
\end{equation}

As the group $G$ is finite and the splitting  $\sigma$ is continuous, the map $\mathfrak{Q}\to \operatorname{Aut}(G)$ by $q\mapsto \operatorname{Ad}_{(e,q)}$ is locally constant. We have assumed that $\Y$ is a banded $G$-gerbe over $\X$. Therefore, there is a continuous (locally constant) map $\varphi: \mathfrak{Q}\to G$ such that 
\[
\operatorname{Ad}_{\varphi(q)}=\operatorname{Ad}_{(e,q)}. 
\]
Eq. (\ref{eq:tau}) implies
\[
\operatorname{Ad}_{\varphi(q_1)}\operatorname{Ad}_{\varphi(q_2)}=\operatorname{Ad}_{\tau(q_1, q_2)}\operatorname{Ad}_{\varphi(q_1q_2)}, 
\]
and $\varphi(q_2)^{-1}\varphi(q_1)^{-1}\tau(q_1, q_2)\varphi(q_1q_2)$ is in the center of $G$. 

Now, consider a new splitting $\sigma': \mathfrak{Q}\to \mathfrak{H}$ by $\sigma': q\mapsto (\varphi(q)^{-1}, q)$.  Then 
\begin{eqnarray*}
\sigma'(q_1)\sigma'(q_2)&=&(\varphi(q_1)^{-1}, q_1)(\varphi(q_2)^{-1}, q_2)=(\varphi(q_1)^{-1}\operatorname{Ad}_{(e, q_1)}(\varphi(q_2)^{-1})\tau(q_1, q_2), q_1q_2)\\ 
&=& (\varphi(q_2)^{-1}\varphi(q_1)^{-1}\tau(q_1,q_2), q_1q_2)\\
&=&\big(\varphi(q_2)^{-1}\varphi(q_1)^{-1}\tau(q_1, q_2)\varphi(q_1q_2), e\big)(\varphi(q_1q_2)^{-1}, q_1q_2)\\
&=&\tau'(q_1,q_2)\sigma'(q_1q_2),
\end{eqnarray*}
where $\tau'(q_1, q_2)=\varphi(q_2)^{-1}\varphi(q_1)^{-1}\tau(q_1, q_2)\varphi(q_1q_2)$ is in the center of $G$. Furthermore, $\operatorname{Ad}_{\sigma'(q)}$ now is
\[
\operatorname{Ad}_{\sigma'(q)}(g)=(\varphi(q)^{-1},q)(g,e)(\tau(q, q^{-1})^{-1}\varphi(q),q^{-1})=(g, e). 
\] 
Therefore, with the new splitting $\sigma'$, the conjugation $\operatorname{Ad}_{\sigma'(q)}$ is the identity automorphism with $\tau'$ taking value in the center of $G$. Furthermore, as $\operatorname{Ad}_{\sigma'(q)}$ is the identity automorphism, Eq. (\ref{eq:tau-cocycle}) implies that $\tau'$ is a $Z(G)$-valued 2-cocycle on $\mathfrak{Q}$.  As the map $\sigma'$ is locally constant, the cocycle $\tau'$ is also locally constant.  It is easy to observe that if $\tau'$ is a coboundary, then the $G$-gerbe is trivial. Therefore, banded $G$-gerbes over $\mathfrak{Q}$ are classified by $H^2(\mathfrak{Q}, Z(G))$. 

Let $Z(G)$ be the center of $G$, and $K$ be the quotient group $G/Z(G)$. We have the following short exact sequence of groups
\begin{equation}\label{eq:center-ext}
1\rightarrow Z(G)\rightarrow G\rightarrow K\rightarrow 1. 
\end{equation}
Therefore, $G$ is a central extension of $K$ by $Z(G)$. Such an extension is classified by the group cohomology $H^2(K, Z(G))$. Let $\nu\in Z^2(K, Z(G))$ be a $Z(G)$-valued 2-cocycle on $K$ determining the extension (\ref{eq:center-ext}). Using $\nu$, we can write an element of $G$  in term of $(z,k)$ such that 
\begin{equation}\label{eq:center-mult}
(z_1, k_1)(z_2, k_2)=(z_1z_2\nu(k_1, k_2), k_1k_2).
\end{equation}
Without loss of generality, we assume that $\nu$ is normalized, i.e. $\nu(1, k)=\nu(k,1)=1$. 


Let $\mathcal{B}$ be an orbifold presented by a proper \'etale groupoid $\mathfrak{Q}\rightrightarrows M$. And let $\Y$ be a banded $G$-gerbe over $\B$. Following the previous discussion, we can present $\Y$ by a proper \'etale groupoid $\mathfrak{H}$ satisfying
\[
M\times G\rightarrow \mathfrak{H}\rightrightarrows M \rightarrow \mathfrak{Q}\rightrightarrows M.
\]
And $\mathfrak{H}$ consists of pairs $(g, q)$ with $g\in G, q\in \mathfrak{Q}$, satisfying
\[
(g_1, q_1)(g_2, q_2)=(g_1g_2 \tau(q_1, q_2), q_1q_2),
\]
where $\tau$ is a $Z(G)$-valued 2-cocycle on $\mathfrak{Q}$. Without loss of generality, we assume that $\tau$ is normalized, i.e. $\tau(1, q)=\tau(q,1)=1$. 

Using the description of $G$ in (\ref{eq:center-ext}) and (\ref{eq:center-mult}), we write $\mathfrak{H}$ 
\[
\mathfrak{H}=G\times_\tau \mathfrak{Q}=(Z(G) \times_\nu K) \times_\tau \mathfrak{Q}=Z(G)\times_{\nu\boxtimes \tau} (K\times \mathfrak{Q}),
\]
where $\nu\boxtimes \tau$ is a $Z(G)$-valued 2-cocycle on $K\times \mathfrak{Q}$ defined as
\[
\nu\boxtimes \tau\big((k_1, q_1), (k_2, q_2)\big):=\nu(k_1, k_2)\tau(q_1, q_2). 
\]
Define a groupoid $\mathfrak{K}\rightrightarrows M$ to be $(K\times \mathfrak{Q})\rightrightarrows M$, and $\mu=\nu\boxtimes \tau$ a $Z(G)$-valued 2-cocycle on $\mathfrak{K}\rightrightarrows M$.  Then we have the following two exact sequences, i.e. 
\begin{eqnarray}
\label{eq:zg-ext} M\times Z(G)\rightarrow \mathfrak{H}=Z(G)\times_\mu \mathfrak{K}\rightarrow \mathfrak{K}\rightrightarrows M,\\
\label{eq:k-ext} M\times K\rightarrow \mathfrak{K}=K\times \mathfrak{Q}\rightarrow \mathfrak{Q}.
\end{eqnarray}
Eq. (\ref{eq:zg-ext}) shows that $\mathfrak{H}$ is a central extension of  $\mathfrak{K}\rightrightarrows M$ by $Z(G)$, determined by the $Z(G)$-valued 2-cocycle $\mu$; and Eq. (\ref{eq:k-ext}) shows that $\mathfrak{K}$ is the product of $\mathfrak{Q}\rightrightarrows M$ with the group $K$. 

\subsection{Discrete torsion}\label{subsec:torsion}
Let $Z:=Z(G)$ be the center of $G$. Following (\ref{eq:center-ext}) and (\ref{eq:center-mult}) of $G$, we write $G=Z\times_{\nu} K$, where $\nu\in Z^2(K, Z)$ is a $Z$-valued 2-cocycle on $K$. Let $\widehat{G}$ be the set of isomorphism classes of irreducible unitary $G$-representations, and $\widehat{Z}$ be the set of isomorphism classes of  irreducible unitary $Z$-representations. $\widehat{Z}$ can be identified with the set of simple characters on $Z$, i.e. group homomorphism $\rho: Z\to U(1)\subset \complex$.  Given $\rho\in \widehat{Z}$, $\rho\circ \nu$ is a $U(1)$-valued 2-cocycle on $K$. Let $\widehat{K}_{\rho\circ\nu}$ be the set of irreducible $\rho\circ \nu$-twisted unitary representations of $K$. We have the following description of $\widehat{G}$. 

\begin{lemma}\label{lem:twisting-rep} There is a natural isomorphism of sets
\[
\widehat{G}\cong \coprod_{\rho \in \widehat{Z}} \widehat{K}_{\rho\circ \nu}. 
\]
\end{lemma}
\begin{proof}
Let $\alpha$ be an irreducible unitary $G$-representation. As $Z$ is the center of $G$, for any element $z$, $\alpha(z)$ commutes with the $G$-representation. Since $\alpha$ is irreducible, $\alpha(z)$ must be a scalar multiple of the identity operator. Hence, the restriction of $\alpha$ to $Z$ is of the form $\alpha(z)=\rho_\alpha (z) Id$,  $\forall z\in Z$, where $\rho_\alpha$ is a character on $Z$.  For $k\in K$, by the identification (\ref{eq:center-mult}), $(1, k)$ is an element of $G$. The following computation shows that the map $\alpha': k\mapsto \alpha(1,k)$ defines a twisted representation of $K$ with respect to the cocycle $\rho_\alpha\circ \nu$, 
\begin{eqnarray*}
\alpha(1,k_1)\alpha(1, k_2)&=&\alpha\big((1, k_1)(1, k_2)\big)\\
&=&\alpha(\nu(k_1, k_2), k_1k_2)\\
&=&\alpha\Big( \big(\nu(k_1, k_2), 1\big)\big(1, k_1k_2\big)\Big)\\
&=&\rho_{\alpha}\big(\nu(k_1, k_2)\big)\alpha(1, k_1k_2). 
\end{eqnarray*}
We observe that an operator $T$ commutes with the $\rho_\alpha\circ \nu$-twisted representation $\alpha'$ of $K$ if and only if $T$ commutes with the representation $\alpha$ of $G$. Therefore, as $\alpha$ is irreducible, so is $\alpha'$. Define the map $\Psi$ by mapping $\alpha$ to $\Psi(\alpha)=\alpha'$. 

Given $\beta\in \widehat{K}_{\rho\circ \nu}$, a $\rho\circ \nu$-twisted irreducible $K$-representation, define a representation $\alpha$ of $G$ by 
\[
\alpha(z,k)=\rho(z)\beta(k).
\] 
The following computation shows that $\alpha$ is a representation of $G$, 
\begin{eqnarray*}
\alpha\big(z_1z_2\nu(k_1, k_2), k_1k_2\big)&=&\rho\big(z_1z_2\nu(k_1, k_2)\big)\beta(k_1k_2)\\
&=&\rho(z_1)\rho(z_2)\rho\circ\nu(k_1, k_2)\beta(k_1k_2)\\
&=&\rho(z_1)\rho(z_2)\beta(k_1)\beta(k_2)\\
&=&\alpha(z_1, k_1)\alpha(z_2, k_2).
\end{eqnarray*}
Observe that an operator $T$ commutes with the representation $\alpha$ of $G$ if and only if $T$ commutes with the $\rho\circ \nu$-twisted $K$-representation $\beta$. As $\beta$ is irreducible, it follows that the representation $\alpha$ is also irreducible. Accordingly, define a map $\Phi: \coprod_{\rho\in \widehat{Z}}\widehat{K}_{\rho\circ \nu}\to \widehat{G}$ by mapping $\beta$ to $\Phi(\beta):=\alpha$.  

It is straight forward to check that $\Psi$ and $\Phi$ are inverse to each other, and provide the desired isomorphism. 
\end{proof}

We extend this discussion to the study of $\mathfrak{H}\rightrightarrows M$.  Following the extension (\ref{eq:zg-ext}), we view $\Y$, which is represented by $\mathfrak{H}$, as a $Z(G)$-gerbe over $\mathcal{Z}$, which is represented by $\mathfrak{K}\rightrightarrows M$. The dual of $\Y$, as a $Z(G)$-gerbe over $\mathcal{Z}$, is $\widehat{\Y}_{\mathcal{Z}}$, which is represented by $\widehat{Z}(G)\times \mathfrak{K}\rightrightarrows \widehat{Z}(G) \times M$. $\widehat{\Y}_{\mathcal{Z}}$ is equipped with a discrete torsion, which is represented by a locally constant $U(1)$-valued 2-cocycle $\hat{c}$ on $\mathfrak{K}$ defined by
\[
\hat{c}\big(([\lambda], k_1, q_1), ([\lambda], k_2, q_2)\big)= \lambda\Big(\mu\big((k_1,q_1),(k_2, q_2)\big)\Big),\qquad [\lambda]\in \widehat{Z}(G).
\]
As $\mathfrak{K}$ is of the product form $K\times \mathfrak{Q}\rightrightarrows M$, with a product type $Z(G)$-valued 2-cocycle $\mu=\nu\boxtimes \tau$, the cocycle $\hat{c}$ is also of product type, i.e. 
\begin{equation}\label{eq:cocycle}
\hat{c}_\lambda=\hat{c}_{\lambda, K}\boxtimes \hat{c}_{\lambda, \mathfrak{Q}},\ \hat{c}_{\lambda, K}(k_1, k_2)=\lambda\big(\nu(k_1, k_2)\big),\ \hat{c}_{\lambda, \mathfrak{Q}}(q_1, q_2)=\lambda\big(\tau(q_1, q_2)\big),
\end{equation}
where $\widehat{c}_{\lambda, K}$ and $\widehat{c}_{\lambda, \mathfrak{Q}}$ are respectively $U(1)$-valued 2-cocycle on $K$ and $\mathfrak{Q}$.  

\subsection{Two theorems}\label{subsec:two-thms}
Inspired by discussions in Sec. \ref{subsec:structure-gerbe} and \ref{subsec:torsion}, we prove the following results about Gromov-Witten invariants. 

Eq. (\ref{eq:zg-ext}) suggests that a banded $G$-gerbe over a smooth proper Deligne-Mumford stack over $\complex$ is a banded $Z(G)$-gerbe. Motivated by this property, we prove the following theorem for GW-invariants on a banded $G$-gerbe when $G$ is abelian.  

\begin{theorem}\label{thm:main-abelian} 
Let $G$ be a finite abelian group, and $\mathfrak{Q}\rightrightarrows M$ be a proper \'etale groupoid. Let $\tau$ be a $G$-valued 2-cocycle on $\mathfrak{Q}\rightrightarrows M$. Define a proper \'etale groupoid $\mathfrak{H}$ to be $G\times_{\tau} \mathfrak{Q}\rightrightarrows M$ with 
\[
(g_1, q_1)(g_2, q_2)=\big(g_1g_2\tau(q_1, q_2), q_1q_2\big). 
\]

Let $\widehat{G}$ be the set of isomorphism classes of irreducible unitary $G$-representations. Define a proper \'etale groupoid $\widehat{\mathfrak{Q}}$ to be $\widehat{G}\times \mathfrak{Q}$ with the unit space $\widehat{G}\times M$.  On $\widehat{\mathfrak{Q}}$, define a $U(1)$-valued $2$-cocycle $c$ by 
\[
c\big((\rho, q_1), (\rho, q_2) \big)=\rho(\tau(q_1, q_2)). 
\]
Denote the restriction of $c$ to the $[\rho]$-component of $\widehat{\mathfrak{Q}}$ by $c_\rho$, a $U(1)$-valued 2-cocycle on the groupoid $\mathfrak{Q}\rightrightarrows M$. 

Let $\Y$ be the Deligne-Mumford stack associated to $\mathfrak{H}$, and $\widehat{\Y}$ be the corresponding DM stack associated to $\widehat{\mathfrak{Q}}$. Let $I: H^\bullet_{CR}(\Y)\to H^\bullet_{CR}(\widehat{\Y}, c)$ be the isomorphism obtained in \cite{ta-ts}. The orbifold $\widehat{\Y}$ is a disjoint union of $\widehat{\Y}_\rho$, which is diffeomorphic to $\B$ equipped with a $U(1)$-valued $2$-cocycle $c_\rho$. For $\delta\in H^\bullet_{CR}(\Y)$, write 
\[
I(\delta)=\sum_{[\rho]\in \widehat{G}}I(\delta)_\rho\in H^\bullet(\widehat{\Y}, c_\rho). 
\]
Then Theorem \ref{thm:main} holds in this setting. More explicitly, let $g\geq 0$, $n>0$, and $a_1,...,a_n\geq 0$ be integers. Let $\beta\in H_2(\B, \mathbb{Q})=H_2(\Y, \mathbb{Q})$ be a curve class. Let $\delta_1,..., \delta_n\in H^\bullet_{CR}(\Y)$. Then we have the following equality of GW invariants:
\[
\<\prod_{j=1}^n \delta_j\psi_j^{a_j}\>_{g,n,\beta}^\Y=\sum_{\rho\in \widehat{G}}\left(\frac{1}{|G|}\right)^{2-2g}\<\prod_{j=1}^n I(\delta_j)_\rho\psi_j^{a_j}\>_{g,n,\beta}^{\B, c_\rho}.
\]
\end{theorem}

\begin{proof}The proof is presented in Sec. \ref{sec:thm-abelian}. 
\end{proof}

Eq. (\ref{eq:k-ext}) suggests that there is a product gerbe naturally associated to a $G$-gerbe on a smooth proper Deligne-Mumford stack. Let $G$ be a finite group. Inspired by the discussion on the cocycle (\ref{eq:cocycle}), we consider a discrete torsion (of a special form) on the product of a smooth proper Deligne-Mumford stack $\B$ with $BG$, and (twisted) Gromov-Witten invariants on the product . 

\begin{theorem}\label{thm:product} 
Let $G$ be a finite group, and $\mathfrak{Q}\rightrightarrows M$ be a proper \'etale groupoid. Let $c_G$ and $c_{\mathfrak{Q}}$ be locally constant $U(1)$-valued 2-cocycles on $G$ and $\mathfrak{Q}\rightrightarrows M$. Let $\widetilde{\mathfrak{Q}}$ be the product groupoid of $G$ and $\mathfrak{Q}$ equipped with the product $2$-cocycle $c_{\widetilde{\mathfrak{Q}}}:=c_G\boxtimes c_{\mathfrak{Q}}$ defined by 
\[
c_{\widetilde{\mathfrak{Q}}}\big((g_1, q_1), (g_2, q_2)\big):=c_{G}(g_1, g_2)c_{\mathfrak{Q}}(q_1, q_2).
\]
Let $\B$ and $\widetilde{\B}$ be the Deligne-Mumford stacks associated with $\mathfrak{Q}$ and $\widetilde{\mathfrak{Q}}$. $\widetilde{\B}$ is the product of $BG$ and $\B$, and $H^\bullet(\widetilde{\B}, c_{\widetilde{\mathfrak{Q}}})$ isomorphic to $H^\bullet(\B, c_{\mathfrak{Q}})\otimes H^\bullet(BG, c_G)$. Let $\widehat{G}_{c_G}$ be the set of isomorphism classes of irreducible $c_G$-twisted  unitary $G$-representations.

Let $g\geq 0$, $n>0$, and $a_1,...,a_n\geq 0$ be integers. Let $\beta\in H_2(\B, \mathbb{Q})$ be a  curve class. 
Let $\delta_1, \cdots, \delta_n\in H_{CR}^\bullet(\widetilde{\B}, c_{\widetilde{\mathfrak{Q}}})$. 
Then there is an equality of twisted Gromov-Witten invariants:
\[
\<\prod_{j=1}^n \delta_j \psi_j^{a_j}\>^{\widetilde{\B}, c_{\widetilde{\mathfrak{Q}}}}_{g,n,\beta}=\sum_{[\rho]\in \widehat{G}_{c_G}}\left(\frac{\text{dim}\, V_\rho}{|G|}\right)^{2-2g}\<\prod_{j=1}^nI(\delta_j) \psi_j^{a_j}\>^{\B, c_{\mathfrak{Q}}}_{g,n,\beta}. 
\]
where $I: H_{CR}^\bullet(\widetilde{\B}, c_{\widetilde{\mathfrak{Q}}})\to\oplus_{[\rho]\in \widehat{G}_{c_G}}H_{CR}^\bullet(\B, c_{\mathfrak{Q}})$ is an isomorphism introduced in \cite{sh-ta-ts}.
\end{theorem}
\begin{proof}
The proof is given in Sec. \ref{sec:pf_thm_prod}. In particular, the isomorphism $I$ is described there.
\end{proof}

\subsection{Proof of Theorem \ref{thm:main}} In this subsection, we use the results developed in Section \ref{subsec:two-thms} to complete the proof of Theorem \ref{thm:main}. 

We start with recalling and summarizing the whole set up.
\begin{enumerate}
\item  Let $\Y$ be a banded $G$-gerbe over a smooth proper Deligne-Mumford stack over $\B$, and $Z(G)$ be the center of $G$. 
\item We represent $\B$ by a proper \'etale groupoid $\mathfrak{Q}\rightrightarrows M$, and represent $\Y$ by the groupoid $G\times_{\tau}\mathfrak{Q}\rightrightarrows M$, where $\tau$ is a smooth $Z(G)$-valued 2-cocycle on $\mathfrak{Q}$. 
\item Let $K$ be the quotient group $G/Z(G)$. Let $\mathfrak{K}\rightrightarrows M$ be the product groupoid $K\times \mathfrak{Q}\rightrightarrows M$, and $\mathcal{Z}$ be the associated Deligne-Mumford stack. 
\item We denote the dual of $\Y$, as a $Z(G)$-gerbe over $\mathcal{Z}$,  by $\widehat{\Y}_{\mathcal{Z}}$, which is represented by $\widehat{Z(G)}\times \mathfrak{K}\rightrightarrows \widehat{Z(G)}\times M$.  $\widehat{\Y}_{\mathcal{Z}}$ can be decomposed as follows, 
\[
\widehat{\Y}_{\mathcal{Z}}=\coprod_{\lambda\in \widehat{Z}(G)} \widehat{\Y}_\mathcal{Z}^\lambda,
\]
where $\widehat{\Y}_\mathcal{Z}^\lambda$ is isomorphic to $\mathcal{Z}$ represented by $\mathfrak{K}\rightrightarrows M$.
\item For each $\lambda\in \widehat{Z}(G)$, there is a $U(1)$-gerbe $\widehat{c}_\lambda$ on $\mathcal{Z}_\lambda=\widehat{\Y}^\lambda_{\mathcal{Z}}$ by Eq. (\ref{eq:cocycle}). On the groupoid $\mathfrak{K}=K\times \mathfrak{Q}\rightrightarrows M$, the $U(1)$-gerbe $\widehat{c}_\lambda$ is represented by $\hat{c}_{\lambda, K}\boxtimes \hat{c}_{\lambda, \mathfrak{Q}}$, where $\widehat{c}_{\lambda,K}$ is a 2-cocycle on $K$ and $\widehat{c}_{\lambda, \mathfrak{Q}}$ is a 2-cocycle on $\mathfrak{Q}$. 
\item Let 
\[
I_{Z(G)}: H^\bullet_{CR}(\Y)\to H^\bullet_{CR}(\widehat{\Y}_{\mathcal{Z}}, \hat{c})=\bigoplus _{\lambda \in \widehat{Z}(G)}H^\bullet_{CR}(\widehat{\Y}^\lambda_{\mathcal{Z}}, \widehat{c}_\lambda),
\] 
be the isomorphism introduced in \cite[Theorem 4.16]{ta-ts}.
\item Let $\widehat{\mathcal{Z}}_\lambda$ be the dual orbifold associated to $(\mathcal{Z}_\lambda, \widehat{c}_{\lambda, K})$ as a trivial $K$-gerbe over $\B$ equipped with the $U(1)$-gerbe $\widehat{c}_{\lambda, K}$. Let $\widehat{K}_\lambda$ be the set of isomorphism classes of $\widehat{c}_{\lambda,K}$-twisted irreducible unitary $K$-representations. $\widehat{\mathcal{Z}}_\lambda$ is decomposed as follows
\[
\widehat{\mathcal{Z}}_\lambda=\coprod_{\nu\in \widehat{K}_\lambda} \B_{\lambda, \nu},
\] 
where $\B_{\lambda, \nu}$ is isomorphic to $\B$ and equipped with a $U(1)$-gerbe $\widehat{c}_{\lambda, \mathfrak{Q}}$. According to Lemma \ref{lem:twisting-rep}, we have the following identification, 
\[
\coprod_{\lambda\in \widehat{Z}(G)}\widehat{\mathcal{Z}}_\lambda=\coprod_{\lambda\in \widehat{Z}(G)}\coprod_{\nu\in \widehat{K}_{\lambda}} \B_{\lambda, \nu}=\coprod_{\rho\in \widehat{G}} \widehat{\Y}_\rho=\widehat{\Y}. 
\] 
\item Let 
\[
I^\lambda_K: H^\bullet_{CR}(\mathcal{Z}_\lambda, \widehat{c}_\lambda)\to H^\bullet_{CR}(\widehat{\mathcal{Z}}_\lambda, \widehat{c}_{\lambda, \mathfrak{Q}})=\bigoplus_{\nu\in \widehat{K}_\lambda} H^\bullet_{CR}(\B_{\lambda, \nu}, \widehat{c}_{\lambda, \mathfrak{Q}})
\]
be the isomorphism introduced in Theorem \ref{thm:product}. 
\end{enumerate}
\begin{lemma}\label{lem:isomorphism} The composition 
\[
I^\lambda_K\circ I_{Z(G)}: H^\bullet_{CR}(\Y)\to \bigoplus_{\lambda \in \widehat{Z}(G)}\bigoplus_{\nu\in \widehat{K}_\lambda} H^\bullet_{CR}(\B_{\lambda, \nu}, \widehat{c}_{\lambda, \mathfrak{Q}})
\]
agrees with the isomorphism 
\[
I_G: H^\bullet_{CR}(\Y)\to \bigoplus_{\rho \in \widehat{G}} H^\bullet_{CR}(\B_\rho, c_\rho)
\]
under the isomorphism $\widehat{G}\cong \coprod_{\lambda \in \widehat{Z}(G)}\widehat{K}_{\lambda}$ in Lemma \ref{lem:twisting-rep}. 
\end{lemma}
\begin{proof}Using the formula of the isomorphism $I$ in \cite[Eq. (4.13)]{ta-ts}, given $\delta\in H^\bullet_{CR}(\Y)$ we have 
\[
I_{Z(G)}(\delta)([\lambda], k, q)=\sum_{z\in Z(G)} \delta(z, k, q)\lambda(z), 
\]
where we have used the property that $Z(G)$ is abelian and $\lambda$ is $1$-dimensional. Similar to $I_{Z(G)}$, we have
\[
I^\lambda_K\big(I_{Z(G)}(\delta)\big)([\lambda], [\nu], q)=\sum_{k\in K}I_{Z(G)}(\delta)([\lambda], k, q)\tr(\nu(k))=\sum_{k\in K, z\in Z(G)}\delta(z, k, q)\tr(\nu(k))\lambda(z).
\]
Observe that $\tr(\nu(k))\lambda(z)=\tr(\lambda(z)\nu(k))$.  By the proof of Lemma \ref{lem:twisting-rep}, $\lambda(z)\nu(k)$ corresponds to the representation $\rho$ of $G$. Therefore, we conclude the following equation,
\[
I^\lambda_K\big(I_{Z(G)}(\delta)\big)([\lambda], [\nu], q)=\sum_{\rho\in \widehat{G}}\delta(z, k, q)\tr(\rho(g))=I_G(\delta)([\rho],q),
\]
which gives the desired identification. 
\end{proof}

\noindent{\bf Proof of Theorem \ref{thm:main}.} Let $\delta_1, ..., \delta_n$ be classes in $H^\bullet_{CR}(\Y)$. By Theorem \ref{thm:main-abelian}, we have 
\[
\<\prod_{j=1}^n \delta_j\psi_j^{a_j}\>_{g,n,\beta}^\Y=\sum_{\lambda \in \widehat{Z}(G)}\left(\frac{1}{|Z(G)|}\right)^{2-2g}\<\prod_{j=1}^n I_{Z(G)}(\delta_j)_\lambda\psi_j^{a_j}\>_{g,n,\beta}^{\widehat{\Y}^\lambda_{\mathcal{Z}},\widehat{c}_\lambda}.
\]
Applying Theorem \ref{thm:product}, we have 
\[
\<\prod_{j=1}^n I_{Z(G)}(\delta_j)_\lambda\psi_j^{a_j}\>_{g,n,\beta}^{\widehat{\Y}^\lambda_{\mathcal{Z}},\widehat{c}_\lambda}=\sum_{\nu \in \widehat{K}_\lambda}\left(\frac{\dim(V_\nu)}{|K|}\right)^{2-2g}\<\prod_{j=1}^n I^\lambda_{K}\big(I_{Z(G)}(\delta_j)_\lambda \big) _{\nu}\psi^{a_j}_j\>_{g,n,\beta}^{\B_{\lambda, \nu}, \widehat{c}_\lambda}
\]
Applying Lemma \ref{lem:twisting-rep} and \ref{lem:isomorphism}, we have
\[
\begin{split}
\<\prod_{j=1}^n \delta_j\psi_j^{a_j}\>_{g,n,\beta}^\Y&=\sum_{\lambda \in \widehat{Z}(G)}\left(\frac{1}{|Z(G)|}\right)^{2-2g}\sum_{\nu \in \widehat{K}_\lambda}\left(\frac{\dim(V_\nu)}{|K|}\right)^{2-2g}\<\prod_{j=1}^n I^\lambda_{K}\big(I_{Z(G)}(\delta_j)_\lambda \big)_\nu \psi^{a_j}_j\>_{g,n,\beta}^{\B_{\lambda, \nu}, \widehat{c}_\lambda}\\
&=\sum_{\rho \in \widehat{G}}\left(\frac{\dim(V_\rho)}{|G|}\right)^{2-2g}\<\prod_{j=1}^n I_G(\delta_j)_\rho \psi^{a_j}_j\>_{g,n,\beta}^{\B, c_\rho},
\end{split}
\]
where we have used the fact that $|G|=|Z(G)||K|$.\ \ \ \ $\hfill{\Box}$
\section{Proof of Theorem \ref{thm:main-abelian}}\label{sec:thm-abelian}
In this section, we present the proof of Theorem \ref{thm:main-abelian}, which is a special case of Theorem \ref{thm:main} when $G$ is abelian. 

\subsection{Formula for the isomorphism $J=I^{-1}$} \label{subsec:I}In this subsection, we briefly recall the definition of the isomorphism $I$ in \cite{ta-ts}.  For the interest of this section, we will assume that $G$ is abelian. 

Let $\widehat{G}$ be the set of isomorphism classes of irreducible unitary representations of $G$.  Because $\Y$ is a banded $G$-gerbe, the dual orbifold $\widehat{\Y}$ can be represented by  $\widehat{G}\times \mathfrak{Q}$. For every isomorphism class $[\rho]\in \widehat{G}$, fix a representative $\rho: G\to V_\rho$, an irreducible unitary representation of $G$. On $\widehat{\Y}$, the flat $U(1)$-gerbe is determined by a $U(1)$-valued 2-cocycle on $c$ on $\widehat{G}\times \mathfrak{Q}$, 
\[
c([\rho], q_1, q_2)=\rho\big(\tau(q_1, q_2)\big). 
\]

In \cite[Sec. 4]{ta-ts}, an isomorphism $I$ from the Chen-Ruan orbifold cohomology of $\Y$ to the $c$-twisted Chen-Ruan orbifold cohomology of $\widehat{\Y}$ is introduced. We recall the explicit formula of $I$ in the case that $\Y$ is a banded $G$-gerbe. 

Let $\mathfrak{Q}^{(0)}$ be the subspace of $\mathfrak{Q}$ consisting of arrows $g\in \mathfrak{Q}$ such that $s(g)=t(g)$. When the $G$-gerbe $\Y$ is banded, $\mathfrak{H}^{(0)}$, the subspace of $\mathfrak{H}^{(0)}$ of arrows with the same source and target, is of the form 
\[
\mathfrak{H}^{(0)}= G\times \mathfrak{Q}^{(0)}. 
\]
The inertia orbifold $I\Y$ is represented by the action groupoid $\mathfrak{H}^{(0)}\rtimes \mathfrak{H}\rightrightarrows \mathfrak{H}^{(0)}$ defined by 
\[
\begin{split}
(g,q)(g_0,q_0) (g, q)^{-1}&=\big(gg_0\tau(q,q_0), qq_0\big)\big(g^{-1}\tau(q,q^{-1})^{-1}, q^{-1}\big)\\
&=\big(gg_0g^{-1}\tau(q,q_0)\tau(qq_0, q^{-1})\tau(q,q^{-1})^{-1}, qq_0q^{-1}\big). 
\end{split}
\]
The Chen-Ruan orbifold cohomology $H^\bullet_{CR}(\Y)$ of $\Y$ is computed by the de Rham cohomology of $\mathfrak{H}$-invariant differential forms $\Omega^\bullet(\mathfrak{H}^{(0)})^{\mathfrak{H}}$. 

The dual orbifold $\widehat{\Y}$ is represented by the groupoid $\widehat{G}\times \mathfrak{Q}$, $|\widehat{G}|$ copies of $\mathfrak{Q}\rightrightarrows M$. The inertia orbifold $I\widehat{\Y}$ is a disjoint union of $|\widehat{G}|$ copies of $\mathfrak{Q}^{(0)}\rtimes \mathfrak{Q}\rightrightarrows \mathfrak{Q}^{(0)}$. For every $[\rho]\in \widehat{G}$,  $\rho(\tau(q_1, q_2))$ is a locally constant $U(1)$-valued 2-cocycle on $\mathfrak{Q}$, and defines a flat line bundle $\call_{[\rho]}$ on $\mathfrak{Q}^{(0)}\rtimes \mathfrak{Q}\rightrightarrows \mathfrak{Q}^{(0)}$, which is a trivial line bundle on $\mathfrak{Q}^{(0)}$ with an action by $\mathfrak{Q}$ defined by
\begin{equation}\label{eq:linegerbe}
\call_{[\rho]} |_{q_0}\ni s \mapsto s\rho\big(
\tau(q,q_0)\tau(qq_0, q^{-1})\tau(q,q^{-1})^{-1}\big)=sc([\rho], q,q_0)c([\rho], qq_0, q^{-1})c([\rho], q, q^{-1})^{-1}.
\end{equation}
This defines a line bundle $\sL_c$ on $\widehat{G}\rtimes \mathfrak{Q}^{(0)}$.
The twisted Chen-Ruan orbifold cohomology of $\widehat{\Y}$ by $c$ is computed by the de Rham cohomology of $\mathfrak{Q}$-invariant $\call_c$-twisted differential forms $\Omega^\bullet(\widehat{G}\rtimes \mathfrak{Q}^{(0)}, \call_c)^{\mathfrak{Q}}$.

In \cite[Sec. 4, Eq. (4.13)]{ta-ts}, a quasi-isomorphism $I: \Omega^\bullet(\mathfrak{H}^{(0)})^{\mathfrak{H}}\to \Omega^\bullet(\widehat{G}\rtimes \mathfrak{Q}^{(0)}, \call_c)^{\mathfrak{Q}}$ is introduced by the following formula
\[
I(\alpha)\big([\rho], q\big)=\sum_g\frac{1}{\dim(V_\rho)}\alpha(g,q)\tr\big( \rho(g){T^{[\rho]}_{q} }^{-1}\big),
\]
for $\alpha\in \Omega^\bullet(\mathfrak{H}^{(0)})^{\mathfrak{H}}$.  The operator $T^{[\rho]}_q: V_\rho \to V_{q([\rho])}$ is the intertwiner between the two representations $\rho\circ \Ad_{\sigma(q)}$ and $V_{q([\rho])}$ satisfying
\[
\rho\big(\Ad_{\sigma(q)}(g)\big)={T^{[\rho]}_q}^{-1}\circ q([\rho])(g)\circ T^{[\rho]}_q. 
\]
As $\Y$ is a banded $G$-gerbe over $\B$, $\Ad_{\sigma(q)}$ is the identity map, and $T^{[\rho]}_q$ can be chosen to be identity map on $V_\rho$. Furthermore, as $\rho$ is  irreducible and $G$ is abelian, $\dim(V_\rho)=1$. Hence, we have the following formula for $I$ in the case that $\Y$ is a banded $G$-gerbe,
\begin{equation}\label{eq:I}
I(\alpha)\big([\rho], q\big)=\sum_g \alpha(g,q)\tr\big( \rho(g)\big). 
\end{equation}

\begin{lemma}\label{lem:mapJ}
The quasi-inverse of the quasi-isomorphism $I: \Omega^\bullet(\mathfrak{H}^{(0)})^{\mathfrak{H}}\to \Omega^\bullet(\widehat{G}\rtimes \mathfrak{Q}^{(0)}, \call_c)^{\mathfrak{Q}}$ is of the form 
\[
J(\beta)(g,q)=\sum_{\rho\in \widehat{G}} \frac{1}{|G|} \pi^*\big(\beta([\rho],q)\big) \tr\big(\rho(g^{-1})\big),
\]
where $\pi: \mathfrak{H}^{(0)}\to \mathfrak{Q}^{(0)}$ is the natural forgetful map from the loop space 
$\mathfrak{H}^{(0)}$ of $\mathfrak{H}$ to the one $\mathfrak{Q}^{(0)}$ of $\mathfrak{Q}$, and $\beta\in \Omega^\bullet(\widehat{G}\rtimes \mathfrak{Q}^{(0)}, \call_c)^{\mathfrak{Q}}$. 
\end{lemma}
\begin{proof} 
Observe that the map $\pi$ is a local diffeomorphism and $\tr(\rho(g^{-1}))$ is locally constant. It follows that the map $I^{-1}$ is compatible with the differential map, i.e.
\[
J\circ d=d\circ J. 
\] 
We compute $I\circ J (\beta)$ to be
\begin{eqnarray*}
&I\Big( \sum_{\rho'\in \widehat{G}} \frac{1}{|G|} \pi^*\big(\beta([\rho],q)\big) \tr\big(\rho'(g^{-1})\big)\Big)([\rho], q)\\
=&\sum_{g}\sum_{\rho'\in \widehat{G}} \tr\big(\rho(g)\big)\frac{1}{|G|} \beta([\rho'],q) \tr\big(\rho'(g^{-1})\big)\\
=&\frac{1}{|G|}\sum_{\rho'\in \widehat{G}}\beta([\rho'],q)\sum_{g}\tr\big(\rho(g)\big)\tr\big(\rho'(g^{-1})\big)\\
=&\beta([\rho],q),
\end{eqnarray*}
where in the last equation, we have used the orthogonality between the two characters of $G$ associated to $\rho$ and $\rho'$.  Similarly, we can compute that $J\circ I=id$. As $I$ is an isomorphism on cohomology, we can conclude that $J$ is the desired inverse of $I$. 
\end{proof}



\begin{remark} A similar argument extends Lemma \ref{lem:mapJ}  to general groups $G$. We leave them to the reader. 

\end{remark}

\subsection{Pushforward of virtual classes}
Let $\Y\to \B$ be a banded $G$-gerbe, with $G$ not necessarily abelian. We use $\mathfrak{l}$ (and $\mathfrak{k}$) to denote a conjugacy class of the stabilizer group of $\Y$ (and $\B$).  Using the conjugacy class $\mathfrak{l}$, we decompose the inertia stacks $I\Y$ and $I\B$ into unions of connected components as follows:
$$I\Y=\coprod_{\mathfrak{l}} \Y_\mathfrak{l}, \quad I\B=\coprod_{\mathfrak{k}}\B_{\mathfrak{k}}.$$
 The map $\pi: \Y\to \B$ induces a map $I\pi: I\Y\to I\B$. For an index $\mathfrak{l}$, define an index $\pi(\mathfrak{l})$ by the requirement that $I\pi$ maps $\Y_{\mathfrak{l}}$ to $\B_{\pi(\mathfrak{l})}$. 
 
 For $\beta\in H_2(\Y, \mathbb{Q})=H_2(\B, \mathbb{Q})$), we use $\Mbar_{g,n}(\Y, \beta)$ (and $\Mbar_{g,n}(\B, \beta)$) to denote the moduli space of $n$-pointed, genus $g$, stable maps to $\Y$ (and to $\B$) in the homology class $\beta$. The map $\pi$ induces a natural map on the corresponding moduli spaces,
\[
  p:\Mbar_{g,n}(\Y, \beta)\to \Mbar_{g,n}(\B,\beta). 
 \]
The $n$-marked points of a curve $\Sigma$ in $\Mbar_{g,n}(\Y, \beta)$ (and $\Mbar_{g,n}(\B,\beta)$) provide a natural evaluation map $ev_\Y$ (and $ev_\B$) to $I\Y^{\times n}$ (and $I\B^{\times n}$).  For $\mathfrak{l}_1, ..., \mathfrak{l}_n$ (and $\mathfrak{k}_1,...,\mathfrak{k}_n$),  we use $\Mbar_{g,n}(\Y, \beta; \mathfrak{l}_1, ..., \mathfrak{l}_n)$ (and $\Mbar_{g,n}(\B, \beta; \mathfrak{k}_1,..., \mathfrak{k}_n)$) to be the level set of the map $ev_\Y$ (and $ev_\B$) of $\Y_{\mathfrak{l}_1}\times ... \times \Y_{\mathfrak{l}_n}$ (and $\B_{\mathfrak{k}_1}\times ...\times \B_{\mathfrak{k}_n}$).

We consider the following commutative diagram between moduli spaces:
\begin{equation}\label{diag:Y_B}
\xymatrix{
\Mbar_{g,n}(\Y, \beta; \mathfrak{l}_1,...,\mathfrak{l}_n)\ar[d]_{p} \ar[r]^{ev_\Y} & \Y_{\mathfrak{l}_1}\times...\times\Y_{\mathfrak{l}_n}\ar[d]^{I\pi^{\times n}}\\
\Mbar_{g,n}(\B, \beta; \pi(\mathfrak{l}_1),...,\pi(\mathfrak{l}_n))\ar[r]_{ev_\B} &\B_{\pi(\mathfrak{l}_1)}\times...\times\B_{\pi(\mathfrak{l}_n)}.
}
\end{equation}

To prove the Eq. (\ref{eqn:decomp_GW_inv}), we will need the following relation between the virtual fundamental classes in Section \ref{subsec:vir_push}.
\begin{lemma}\label{lem:pushforward} 
There are  rational numbers $d_{\vec{\mathfrak{l}}, \mathfrak{c}}$ satisfying 
\begin{equation}\label{eqn:vir_push}
p_*[\Mbar_{g,n}(\Y, \beta; \mathfrak{l}_1,...,\mathfrak{l}_n)]^{vir}=\sum _{\mathfrak{c}}d_{\overrightarrow{\mathfrak{l}},\mathfrak{c}}[\Mbar_{g,n}(\B, \beta; \pi(\mathfrak{l}_1),...,\pi(\mathfrak{l}_n))_\mathfrak{c}]^{vir},
\end{equation}
where $\mathfrak{c}$ runs over all irreducible cycles of the virtual class $[\Mbar_{g,n}(\B, \beta; \pi(\mathfrak{l}_1),...,\pi(\mathfrak{l}_n))]^{vir}$. 
\end{lemma}
The description of the number $d_{\overrightarrow{\mathfrak{l}},\mathfrak{c}}$ is given in Section \ref{sec:degree}. The proof of Lemma \ref{lem:pushforward} is given in Section \ref{subsec:vir_push}.

\subsection{Proof of Theorem \ref{thm:main-abelian}}\label{subsec:pf_outline}
In this subsection, we present the proof of Theorem \ref{thm:main-abelian}.  In this subsection we assume that $G$ is abelian. Fix $\rho\in \widehat{G}$,  and denote the inner local system supported on the component $[\rho]\times \B\subset \widehat{\Y}$ by $\mathcal{L}_{c_\rho}$.  Let $I\B$ be the inertia orbifold associated to $\B$, and $\B_{\mathfrak{k}}\subset I\B$ be a component of $I\B$. For $j=1, \cdots, n$, consider $\varphi_j\in H^\bullet(\B_{\mathfrak{k}_j}, \mathcal{L}_{c_\rho} )$.  By Lemma \ref{lem:mapJ}, $J(\varphi_j)$ is a class  in $H^\bullet_{CR}(\Y):=H^\bullet\big(\Omega^\bullet(\mathfrak{H}^{(0)})^{\mathfrak{H}}\big)$ such that $I\big(J(\varphi_j)\big)=\varphi_j$. 

We have the following two lemmas about the trivialization $\theta_{c_\rho*}$ and degree $d_{\mathfrak{\vec{\mathfrak{l}}},\mathfrak{c}}$. 

\begin{lemma}\label{lem:hol} 
 Let $\Sigma\to \B$ be a point of $\Mbar_{g,n}(\B, \beta; \mathfrak{l}_1,...,\mathfrak{l}_n)$, and let $\Y_\Sigma\to \Sigma$ be the pull back of $\Y\to \B$ via $\Sigma\to \B$, and let $(g_1, ..., g_n)\in G^{\times n}$ such that $\big((g_1, q_1), ..., (g_n,q_n)\big)$ (and $(q_1, ..., q_n)$) are the generators of the stabilizer groups of $\Y_\Sigma$ over $\Sigma$ (and $\Sigma$) at the $n$-marked points. Let $\tau$ be a $Z(G)$-valued 2-cocycle representing the $G$-gerbe $\Y$ over $\mathcal{B}$. And the pullback of $\tau$ to $\Sigma$ is a $Z(G)$-valued 2-cocycle representing the $G$-gerbe $\Y_\Sigma$ over $\Sigma$.  Denote the holonomy of $\tau$ around the $i$-th marked point $z_i$ as is introduced in Definition \ref{dfn:holonomy-L} by $hol(\Y_\Sigma, \tau)(L_i, z_i)$ associated to $\Y_\Sigma$. When $G$ is abelian, 
\[
\prod_{j=1}^n \tr\big(\rho(g_j)\big)=\prod_{i=1}^n \rho\Big(\big(hol(\Y_\Sigma, \tau)(L_i, z_i)\big)^{-1}\Big).
\]
Furthermore, $\prod_{i=1}^n \rho\Big(\big(hol(\Y_\Sigma, \tau)(L_i, z_i)\big)^{-1}\Big)$ is the trivializing constant $\theta_{c_{\rho}}$ for the line bundle $ev_\B^*\big(\mathcal{L}_{c_\rho}^{\boxtimes n}\big)$ at $\Sigma$ in $\Mbar_{g,n}(\B, \pi(\beta); \pi(\mathfrak{l}_1),...,\pi(\mathfrak{l}_n))$. 
\end{lemma}
\begin{proof}
According to Theorem \ref{thm:degree}, the holonomy $\mathfrak{c}=\prod_{i=1}^n hol(\Y_\Sigma, \tau)(L_i, z_i)$ and the local stabilizer data $g_i$ satisfy
\[
[x_1, y_1]\cdots [x_g, y_g]=\mathfrak{c} g_1\cdots g_n. 
\]
As $\rho\in \widehat{G}$ is $1$-dimensional, $\rho([x_j, y_j])=1$ for $j=1, \cdots, g$, and we have 
\[
\rho_\mathfrak{c}\left(\prod_{j=1}^n \chi_\rho(g_j)\right)=1. 
\] 

The trivializing constant for the line bundle $ev_\B^*\big(\mathcal{L}_{c_\rho}^{\boxtimes n}\big)$ is equal to 
\[
\prod_{i=1}^n \big(hol(\Y_\Sigma, c_\rho)(L_i, z_i)\big)^{-1}, 
\]
where $c_\rho$ is the line gerbe on $\B$ defined by $\rho(\tau)$. Using the properties that $G$ is abelian, and that $\rho$ is a group homomorphism, we can directly check
\[
\prod_{i=1}^n \big(hol(\Y_\Sigma, c_\rho)(L_i, z_i)\big)^{-1}=\prod_{i=1}^n \rho\Big(\big(hol(\Y_\Sigma, \tau)(L_i, z_i)\big)^{-1}\Big),
\]
which proves the desired statement. 
\end{proof}

\begin{lemma}\label{lem:degree} When $G$ is abelian, $d_{\vec{\mathfrak{l}},\mathfrak{c}}$ is computed to be 
\[
\left(\frac{1}{|G|}\right)^{1-2g}\#\{(g_1)\in \mathsf{l}_1, \cdots, (g_n)\in \mathsf{l}_n, g_1\cdots g_n \mathfrak{c}=1\},
\]
where $\mathsf{l}_1, ..., \mathsf{l}_n$ are collections of conjugacy classes of $G$ introduced in Lemma \ref{lem:conjugacy}.
\end{lemma}

\begin{proof}
By Eq. (\ref{eq:degree}) in Theorem \ref{thm:char-formula}, we compute the degree $d_{\vec{\mathfrak{l}}, \mathfrak{c}}$ by 
\[
d_{\vec{\mathfrak{l}}, \mathfrak{c}}=\sum_{(g_1)\in \mathsf{l}_1, \cdots, (g_n)\in \mathsf{l}_n}\sum_{\alpha \in \widehat{G}}\frac{1}{\dim(V_\alpha)^n}\left(\frac{\text{dim}\, V_\alpha}{|G|}\right)^{2-2g}\alpha_\mathfrak{c}\left(\sum_{g^\bullet_1\in (g_1),...,g^\bullet_n\in (g_n)}\prod_{j=1}^n \chi_\alpha(g^\bullet_j)\right),
\]
where $\mathfrak{c}$ is  $\prod_{i=1}^n hol(\Y_\Sigma, \tau)(L_i, z_i)$. 

When $G$ is abelian, all $\dim(V_\alpha)=1$, and we can simplify the above expression to be
\[
\begin{split}
d_{\vec{\mathfrak{l}}, \mathfrak{c}}=&\sum_{(g_1)\in \mathsf{l}_1, \cdots, (g_n)\in \mathsf{l}_n}\sum_{\alpha \in \widehat{G}}\left(\frac{1}{|G|}\right)^{2-2g}\alpha_\mathfrak{c}\left(\sum_{g^\bullet_1\in (g_1),...,g^\bullet_n\in (g_n)}\prod_{j=1}^n \chi_\alpha(g^\bullet_j)\right)\\
=&\left(\frac{1}{|G|}\right)^{2-2g}\sum_{(g_1)\in \mathsf{l}_1, \cdots, (g_n)\in \mathsf{l}_n}\sum_{\alpha \in \widehat{G}}\alpha_\mathfrak{c}\left(\sum_{g^\bullet_1\in (g_1),...,g^\bullet_n\in (g_n)}\prod_{j=1}^n \chi_\alpha(g^\bullet_j)\right).
\end{split}
\]
As $G$ is abelian, each conjugacy class of $G$ contains only one element. Therefore, we can further simplify the expression of $d_{\vec{\mathfrak{l}},\mathfrak{c}}$ to be
\[
\left(\frac{1}{|G|}\right)^{2-2g}\sum_{(g_1)\in \mathsf{l}_1, \cdots, (g_n)\in \mathsf{l}_n}\sum_{\alpha \in \widehat{G}}\alpha_\mathfrak{c}\left(\prod_{j=1}^n \chi_\alpha(g^\bullet_j)\right). 
\]

By the same argument as Lemma \ref{lem:hol}, we have 
\[
\alpha_\mathfrak{c}\left(\prod_{j=1}^n \chi_\alpha(g^\bullet_j)\right)=1. 
\] 
Therefore, we can simplify the expression of $d_{\vec{\mathfrak{l}},\mathfrak{c}}$ to be
\[
\left(\frac{1}{|G|}\right)^{2-2g}\sum_{\begin{array}{c}(g_1)\in \mathsf{l}_1, \cdots, (g_n)\in \mathsf{l}_n, \\g_1\cdots g_n \mathfrak{c}=1\end{array}}\sum_{\alpha \in \widehat{G}} 1.
\]
Using the fact that $|\widehat{G}|=|G|$, we have 
\[
d_{\vec{\mathfrak{l}}, \mathfrak{c}}=\left(\frac{1}{|G|}\right)^{1-2g}\sum_{\begin{array}{c}(g_1)\in \mathsf{l}_1, \cdots, (g_n)\in \mathsf{l}_n, \\g_1\cdots g_n \mathfrak{c}=1\end{array}} 1, 
\]
which gives the desired expression for $d_{\vec{\mathfrak{l}}, \mathfrak{c}}$. 
\end{proof}

By definition, we can write the left hand side of Eq. (\ref{eqn:decomp_GW_inv}) as 

\begin{equation*}
\<\prod_{j=1}^n J(\varphi_j) \psi_j^{a_j}\>_{g,n,\beta}^\Y=\sum_{\mathfrak{l}_1,...,\mathfrak{l}_n}\text{pt}_*\Big(ev_\Y^*\big(\bigotimes_{j=1}^nJ(\varphi_j)\big)(\prod_{j=1}^n\psi_j^{a_j})\cap [\Mbar_{g,n}\big(\Y, \beta; \mathfrak{l}_1,...,\mathfrak{l}_n\big)]^{vir}\Big)
\end{equation*}
where $\text{pt}_*$ stands for pushing forward to a point.   Fix $\mathfrak{l}_1, ..., \mathfrak{l}_n$. By  Lemma \ref{lem:mapJ}, we may compute the left side of the above equation as follows: 
\[
\begin{split}
&\text{pt}_*\Big(ev_\Y^*\big(\bigotimes_{j=1}^nJ(\varphi_j)\big)(\prod_{j=1}^n\psi_j^{a_j})\cap [\Mbar_{g,n}\big(\Y, \beta; \mathfrak{l}_1,...,\mathfrak{l}_n\big)]^{vir}\Big)\\
=&\text{pt}_*\Big( ev_\Y^* \big( \frac{1}{|G|} \tr\big(\rho(g_j)\big)I\pi^*\big(\varphi_j([\rho], q_j)\big)\big) (\prod_{j=1}^n \psi_j^{a_j})\\
&\qquad\qquad\qquad\qquad\qquad\qquad\qquad\qquad\qquad \cap [\Mbar_{g,n}\big(\Y, \beta; \mathfrak{l}_1,...,\mathfrak{l}_n\big)]^{vir} \Big)\\
=&\text{pt}_*\Big(\Big(\frac{1}{|G|}\Big)^n\prod_{j=1}^n \tr\big(\rho(g_j)\big)  ev_\Y^* \big( \bigotimes_{j=1}^n I\pi^*\big(\varphi_j([\rho], q_j)\big)\big) (\prod_{j=1}^n \psi_j^{a_j})\\
&\qquad \qquad \qquad \qquad \qquad \qquad \qquad  \qquad\qquad\cap [\Mbar_{g,n}\big(\Y, \beta; \mathfrak{l}_1,...,\mathfrak{l}_n\big)]^{vir} \Big)
\end{split}.
\]
By the diagram (\ref{diag:Y_B}), we have 
\[
ev_\Y^*\circ (I\pi^{\times n})^*= p^* \circ ev_\B^* .
\]
And we can continue the above computation by 
\[
\begin{split}
=& \text{pt}_*\Big( \Big(\frac{1}{|G|}\Big)^np^*\big( ev_{\B}^* \big( \prod_{j=1}^n \tr\big(\rho(g_j)\big) \bigotimes_{j=1}^n \varphi_j([\rho], q_j)\big) \big) (\prod_{j=1}^n \psi_j^{a_j})\\
&\qquad \qquad \qquad \qquad \qquad \qquad \cap [\Mbar_{g,n}\big(\Y, \beta; \mathfrak{l}_1,...,\mathfrak{l}_n\big)]^{vir} \Big)\\
=& \text{pt}_*\Big( \Big(\frac{1}{|G|}\Big)^n p^*\big( ev_{\B}^*\big( \prod_{j=1}^n \tr\big(\rho(g_j)\big)\bigotimes_{j=1}^n \varphi_j([\rho], q_j)\big)\wedge (\prod_{j=1}^n \psi_j^{a_j}) \big)\\
&\qquad \qquad \qquad \qquad \qquad \qquad \cap [\Mbar_{g,n}\big(\Y, \beta; \mathfrak{l}_1,...,\mathfrak{l}_n\big)]^{vir} \Big),
\end{split}
\]
where in the last equation, we have used the functorial property of the descendants $\psi_j$.

By the pushforward of integral, we continue the above computation by
\[
\begin{split}
=&\text{pt}_*\Big(  \Big(\frac{1}{|G|}\Big)^n\prod_{j=1}^n \tr\big(\rho(g_j)\big) ev_{\B}^*\big( \bigotimes_{j=1}^n \varphi_j([\rho], q_j)\big)(\prod_{j=1}^n \psi_j^{a_j}) \\
&\qquad \qquad \qquad \qquad \qquad \qquad \cap p_*[\Mbar_{g,n}\big(\Y, \beta; \pi(\mathfrak{l}_1),...,\pi(\mathfrak{l}_n)\big)]^{vir} \Big).
\end{split}
\]

We apply Eq. (\ref{eqn:vir_push}) in Lemma \ref{lem:pushforward} to compute the above expression. 
\begin{equation}\label{eq:abc}
\begin{split}
=& \text{pt}_*\Big(  \Big(\frac{1}{|G|}\Big)^n\prod_{j=1}^n \tr\big(\rho(g_j)\big) ev_{\B}^*\big( \bigotimes_{j=1}^n \varphi_j([\rho], q_j)\big)(\prod_{j=1}^n \psi_j^{a_j})\\
&\qquad \qquad \qquad \qquad \qquad \qquad \cap \sum_\mathfrak{c} d_{\vec{\mathfrak{l}}, \mathfrak{c}}[\Mbar_{g,n}\big(\B, \beta; \pi(\mathfrak{l}_1),...,\pi(\mathfrak{l}_n)\big)_\mathfrak{c}]^{vir} \Big)\\
=&\sum_{\mathfrak{c}} \text{pt}_* \Big( d_{\vec{\mathfrak{l}}, \mathfrak{c}}\Big(\frac{1}{|G|}\Big)^n\prod_{j=1}^n \tr\big(\rho(g_j)\big) ev_{\B}^*\big( \bigotimes_{j=1}^n \varphi_j([\rho], q_j)\big)(\prod_{j=1}^n \psi_j^{a_j}) \\
&\qquad \qquad \qquad \qquad \qquad \qquad \cap[\Mbar_{g,n}\big(\B, \beta; \pi(\mathfrak{l}_1),...,\pi(\mathfrak{l}_n)\big)_\mathfrak{c}]^{vir} \Big).
\end{split}
\end{equation}

Using Lemma \ref{lem:hol} and \ref{lem:degree}, we can write 
\begin{equation}\label{eq:hol-deg}
\begin{split}
&d_{\vec{\mathfrak{l}},\mathfrak{c}}\Big(\frac{1}{|G|}\Big)^n\prod_{j=1}^n \tr\big(\rho(g_j)\big) ev_{\B}^*\big( \bigotimes_{j=1}^n \varphi_j([\rho], q_j)\big)\\
=&\left(\frac{1}{|G|}\right)^{n+1-2g}\#\{(g_1)\in \mathsf{l}_1, \cdots, (g_n)\in \mathsf{l}_n, g_1\cdots g_n \mathfrak{c}=1\} \theta_{c_\rho*}\Big(ev_\B^*\big( \bigotimes_{j=1}^n \varphi_j([\rho], q_j)\big)\Big).
\end{split}
\end{equation}
Substituting Eq. (\ref{eq:hol-deg})  back to Eq. (\ref{eq:abc}), we have 
\begin{equation}\label{eq:sum}
\begin{split}
&\sum_{\mathfrak{c}}\text{pt}_* \Big(d_{\vec{\mathfrak{l}}, \mathfrak{c}}\Big(\frac{1}{|G|}\Big)^n\prod_{j=1}^n \tr\big(\rho(g_j)\big) ev_{\B}^*\big( \bigotimes_{j=1}^n \varphi_j([\rho], q_j)\big)(\prod_{j=1}^n \psi_j^{a_j})\\
&\qquad\qquad\qquad\qquad \qquad \qquad \cap[\Mbar_{g,n}\big(\B, \beta; \pi(\mathfrak{l}_1),...,\pi(\mathfrak{l}_n)\big)_\mathfrak{c}]^{vir} \Big)\\
=&\left(\frac{1}{|G|}\right)^{n+1-2g} \sum_{\mathfrak{c}} \#\{(g_1)\in \mathsf{l}_1, \cdots, (g_n)\in \mathsf{l}_n, g_1\cdots g_n \mathfrak{c}=1\} 
\\
&\qquad \qquad \text{pt}_* \left\{\theta_{c_\rho*}\Big(ev_\B^*\big( \bigotimes_{j=1}^n \varphi_j([\rho], q_j)\big)\Big)(\prod_{j=1}^n \psi_j^{a_j}) \cap[\Mbar_{g,n}\big(\B, \beta; \pi(\mathfrak{l}_1),...,\pi(\mathfrak{l}_n)\big)_{\mathfrak{c}}]^{vir} \Big)\right\}.
\end{split}
\end{equation}
Taking the sum over all possible $\vec{\mathfrak{l}}$ with fixed $\pi(\mathfrak{l}_1),...,\pi(\mathfrak{l}_n)$,  we have
\[
 \sum_{\vec{\mathfrak{l}}} \#\{(g_1)\in \mathsf{l}_1, \cdots, (g_n)\in \mathsf{l}_n, g_1\cdots g_n c=1\} =\#\{g_1\cdots g_n c=1\}=|G|^{n-1}. 
\]
We summarize the above computation to the following equality.
\begin{equation*}
\begin{split}
\<\prod_{j=1}^n J(\varphi_j) \psi_j^{a_j}\>_{g,n,\beta}^\Y&=\sum_{\mathfrak{l}_1,...,\mathfrak{l}_n} \text{pt}_*\Big(ev_\Y^*\big(\bigotimes_{j=1}^nJ(\varphi_j)\big)(\prod_{j=1}^n\psi_j^{a_j})\cap [\Mbar_{g,n}\big(\Y, \beta; \mathfrak{l}_1,...,\mathfrak{l}_n\big)]^{vir}\Big)\\
&=\left(\frac{1}{|G|}\right)^{2-2g}\sum_{\mathfrak{c}}\text{pt}_* \Big\{\theta_{c_\rho*}\Big(ev_\B^*\big( \bigotimes_{j=1}^n \varphi_j([\rho], q_j)\big)\Big)(\prod_{j=1}^n \psi_j^{a_j}) \\
&\qquad \qquad \qquad \cap[\Mbar_{g,n}\big(\B, \beta; \pi(\mathfrak{l}_1),...,\pi(\mathfrak{l}_n)\big)_{\mathfrak{c}}]^{vir} \Big)\Big\}\\
&=\left(\frac{1}{|G|}\right)^{2-2g}\text{pt}_* \Big\{\theta_{c_\rho*}\Big(ev_\B^*\big( \bigotimes_{j=1}^n \varphi_j([\rho], q_j)\big)\Big)(\prod_{j=1}^n \psi_j^{a_j}) \\
&\qquad \qquad \qquad \cap\big(\sum_{\mathfrak{c}} [\Mbar_{g,n}\big(\B, \beta; \pi(\mathfrak{l}_1),...,\pi(\mathfrak{l}_n)\big)_{\mathfrak{c}}]^{vir}\big) \Big)\Big\}\\
&=\left(\frac{1}{|G|}\right)^{2-2g}\<\prod_{j=1}^n \varphi_j \psi_j^{a_j}\>_{g,n,\beta}^{\B, c_\rho},
\end{split}
\end{equation*}
which completes the proof of Theorem \ref{thm:main-abelian}. 

\section{Proof of Theorem \ref{thm:product}}\label{sec:pf_thm_prod}

Let $\B$ be a smooth proper Deligne-Mumford stack over $\mathbb{C}$ represented by a proper \'etale groupoid $\mathfrak{Q}\rightrightarrows M$, and $G$ a finite group. We have the following diagram
\begin{equation*}
\xymatrix{
\B\times BG\ar[d]_{\pi_1}\ar[r]^{\pi_2} & BG\\
\B. &
}
\end{equation*}
Let $c_1$ be a flat $U(1)$-gerbe on $\B$ and $c_2$ a flat $U(1)$-gerbe on $BG$ represented by a $U(1)$-valued 2-cocycle on $\B$ and $G$. Define $c_1\boxtimes c_2:=\pi_1^*c_1\otimes \pi_2^* c_2$ to be a flat $U(1)$-gerbe on $\B\times BG$. The purpose of this section is to study the $c_1\boxtimes c_2$-twisted Gromov-Witten theory of $\B\times BG$. 

Given a stable map $f: \C\to \B\times BG$, stabilizations of the compositions $\pi_1\circ f$ and $\pi_2\circ f$ yield stable maps $f_1: \C_1\to \B$ and $f_2: \C_2\to BG$. This construction gives a  map between moduli spaces of stable maps:
\begin{equation}\label{eqn:map_modulispaces}
p: \Mbar_{g,n}(\B\times BG, \beta)\to \Mbar_{g,n}(\B, \beta)\times \Mbar_{g,n}(BG).
\end{equation}

Let $\sL_{c_1}\to I\B$, $\sL_{c_2}\to IBG$, and $\sL_{c_1\boxtimes c_2}\to I(\B\times BG)$ be the inner local systems associated to $c_1$, $c_2$, and $c_1\boxtimes c_2$ respectively, as constructed in \cite{pry} and recalled in (\ref{eq:linegerbe}). The definition of gerbe-twisted Gromov-Witten invariants in \cite{pry} involves trivializations of certain line bundles. More precisely, there are trivializations
\begin{equation}
\begin{cases}
\theta_{c_1}: \otimes_{i=1}^n ev_i^*\sL_{c_1}\to \underline{\mathbb{C}} \quad \text{ on } \Mbar_{g,n}(\B, \beta)\\
\theta_{c_2}: \otimes_{i=1}^n ev_i^*\sL_{c_2}\to \underline{\mathbb{C}} \quad \text{ on } \Mbar_{g,n}(BG)\\
\theta_{c_1\boxtimes c_2}: \otimes_{i=1}^n ev_i^*\sL_{c_1\boxtimes c_2}\to \underline{\mathbb{C}} \quad \text{ on } \Mbar_{g,n}(\B\times BG, \beta).
\end{cases}
\end{equation}
Using the constructions of these trivializations in \cite{pry}, it is straightforward to check that the following diagram is commutative:
\begin{equation}\label{diagram:prod_triv}
\xymatrix{
\otimes_{i=1}^n ev_i^*\sL_{c_1\boxtimes c_2}\ar[r]^{\,\,\,\,\,\,\,\, \theta_{c_1\boxtimes c_2}}  &{\mathbb{C}}\\ 
p^*\left((\otimes_{i=1}^n ev_i^*\sL_{c_1})\boxtimes (\otimes_{i=1}^n ev_i^*\sL_{c_2}) \right)\ar[u]^{\simeq}\ar[ur]_{\theta_{c_1}\boxtimes \theta_{c_2}} &
}.
\end{equation}
Let $\alpha_1,...,\alpha_n\in H^\bullet(I\B, \mathcal{L}_{c_1})$, $a_1,...,a_n\in \mathbb{Z}_{\geq 0}$. Let $C^*( G, c_2)$ be the $c_2$-twisted group algebra which is spanned by group elements of $G$ with the following product rule, 
\[
g_1 \circ g_2:=c_2(g_1, g_2)g_1g_2. 
\]
Let $Z(C^*(G, c_2))$ be the center of the twisted group algebra $C^*(G, c_2)$. Consider a conjugacy class $(g)$ of $G$. Let $1_{(g)}$ be the corresponding element in $C^*(G, c_2)$ of the form 
\[
1_{(g)}=\sum_{g'\in (g)}g'. 
\] 
It is not hard to check that $1_{(g)}$ is in the center $Z(C^*(G, c_2))$ of $C^*(G, c_2)$, if and only if $c_2(g, g')=c_2(g',g)$ for any $g'\in G$ commuting with $g$. A conjugacy class $(g)$ with such a property is called $c_2$-{\em regular}.  It is explained \cite[Examples 6.4]{ru1} that $Z(C^*(G, c_2))$ has an additive basis consisting of $1_{(g)}$ that $(g)$ is $c_2$-regular. Consider $(g_1),...,(g_n)\subset G$ such that $1_{(g_i)}$ is in the center of $Z(C^*(G, c_2))$ for $i=1,...,n$. The $c_1\boxtimes c_2$-twisted Gromov-Witten invariants of $\B\times BG$, 
\begin{equation}\label{eqn:gw_c1c2}
\<\prod_{i=1}^n \tau_{a_i}(\alpha_i\otimes 1_{(g_i)})\>_{g,n,\beta}^{\B\times BG, c_1\boxtimes c_2}:=\int_{[\Mbar_{g,n}(\B\times BG, \beta)]^{vir}} \left(\theta_{c_1\boxtimes c_2}\right)_* \left(\prod_{i=1}^n ev_i^*(\alpha_i\otimes 1_{(g_i)}) \right)\left(\prod_{i=1}^n\psi_i^{a_i} \right),
\end{equation} 
may be evaluated as follows. 
\begin{proposition}
\begin{equation}\label{eqn:gw_c1c2_2}
\begin{split}
&\<\prod_{i=1}^n \tau_{a_i}(\alpha_i\otimes 1_{(g_i)})\>_{g,n,\beta}^{\B\times BG, c_1\boxtimes c_2}\\
=&\<\prod_{i=1}^n\tau_{a_i}(\alpha_i)\>_{g,n,\beta}^{\B, c_1}\times \Lambda_{g,n}^{BG, c_2}(1_{(g_1)},..., 1_{(g_n)}),
\end{split}
\end{equation}  
where $\Lambda_{g,n}^{BG, c_2}$ is defined and determined in Proposition \ref{prop:cohft} in Appendix \ref{app:twisted_bgamma}. 
\end{proposition}
\begin{proof}
The virtual fundamental classes $[\Mbar_{g,n}(\B\times BG, \beta)]^{vir}$, $[\Mbar_{g,n}(\B, \beta)]^{vir}$, and $[\Mbar_{g,n}(BG)]$ are explicitly related. This is a special case of the main result of \cite{ajt1} and is somewhat complicated to state. However, given the commutativity of (\ref{diagram:prod_triv}), the invariant (\ref{eqn:gw_c1c2}) may be processed in the same way as \cite[Section 5]{ajt1}, which gives the proposition.

An alternative proof of (\ref{eqn:gw_c1c2_2}) is as follows. Write $p_1$ and $p_2$ for the projection map from $\Mbar_{g,n}(\B, \beta)\times \Mbar_{g,n}(BG, (g_1),...,(g_n))$ to the first and second factors respectively. Then it is straightforward to verify that 
$$ev_i^*(\alpha_i\otimes 1_{(g_i)})=p^*(p_1^*ev_i^*\alpha_i\otimes p_2^*ev_i^*1_{(g_i)}),$$
where we have abused notations and denote the evaluation maps on all moduli spaces by $ev_i$. Also, by definition of descendant classes, we have $$\psi_i=p^*p_1^*\psi_i,$$
where, again by abuse of notation, $\psi_i$ on the left-hand side is on $\Mbar_{g,n}(\B\times BG, \beta)$ and $\psi_i$ on the right-hand side is on $\Mbar_{g,n}(\B, \beta)$. Diagram (\ref{diagram:prod_triv}) reads $$p^*(\theta_{c_1}\boxtimes\theta_{c_2})=\theta_{c_1\boxtimes c_2}.$$
Therefore 
\begin{equation*}
\begin{split}
&\left(\theta_{c_1\boxtimes c_2}\right)_* \left(\prod_{i=1}^n ev_i^*(\alpha_i\otimes 1_{(g_i)}) \right)\left(\prod_{i=1}^n\psi_i^{a_i} \right)\\
=&p^*\left(p_1^*\left(\theta_{c_1}\right)_*\left( \prod_{i=1}^nev_i^*\alpha_i\right) p_2^*\left(\theta_{c_2} \right)_*\left(\prod_{i=1}^n ev_i^*1_{(g_i)} \right) \right) p^*p_1^*\left(\prod_{i=1}^n\psi_i^{a_i}\right).
\end{split}
\end{equation*}
In the notation of Appendix \ref{app:twisted_bgamma}, $\left(\theta_{c_2} \right)_*\left(\prod_{i=1}^n ev_i^*1_{(g_i)} \right)=:\theta_{c_2}(1_{(g_1)},...,1_{(g_n)})$ is in $\mathbb{C}$. Then we get 
\begin{equation*}
\begin{split}
&\<\prod_{i=1}^n \tau_{a_i}(\alpha_i\otimes 1_{(g_i)})\>_{g,n,\beta}^{\B\times BG, c_1\boxtimes c_2}\\
=&\text{pt}_*\left( \left(\theta_{c_1\boxtimes c_2}\right)_* \left(\prod_{i=1}^n ev_i^*(\alpha_i\otimes 1_{(g_i)}) \right)\left(\prod_{i=1}^n\psi_i^{a_i} \right)\cap [\Mbar_{g,n}(\B\times BG, \beta)]^{vir}\right)\\
=&\text{pt}_*\left(\left(\theta_{c_1}\right)_*\left( \prod_{i=1}^nev_i^*\alpha_i\right) \prod_{i=1}^n\psi_i^{a_i}\cap (p_1p)_*[\Mbar_{g,n}(\B\times BG, \beta)]^{vir}\right)\cdot \theta_{c_2}(1_{(g_1)},...,1_{(g_n)}).
\end{split}
\end{equation*}
By the virtual pushforward formula, Lemma \ref{lem:pushforward}, we have $$ (p_1p)_*[\Mbar_{g,n}(\B\times BG, \beta)]^{vir}=\Omega_g^G((g_1),...,(g_n))[\Mbar_{g,n}(\B, \beta)]^{vir}.$$
The result follows by the definition of $\Lambda_{g, n}^{BG, c_2}$ in Appendix \ref{app:twisted_bgamma}.
\end{proof}

Let $[\rho_1],..., [\rho_n]\in \widehat{G}_{c_2}$ and let $f_{\rho_1},..., f_{\rho_n}$ be as in (\ref{eqn:idempotent}). Then (\ref{eqn:gw_c1c2_2}) together with Theorem \ref{thm:inv} gives
\begin{equation}\label{eqn:gw_c1c2_3}
\begin{split}
&\<\prod_{i=1}^n \tau_{a_i}(\alpha_i\otimes f_{\rho_i})\>_{g,n,\beta}^{\B\times BG, c_1\boxtimes c_2}\\
=&\begin{cases}
\left(\frac{\text{dim } V_\rho}{|G|}\right)^{2-2g} \<\prod_{i=1}^n\tau_{a_i}(\alpha_i)\>_{g,n,\beta}^{\B, c_1} \quad \text{ if } \rho_1=...=\rho_n=:\rho\\
0 \quad \text{ else}.
\end{cases}
\end{split}
\end{equation}  
It is straightforward to check that the map $a\mapsto a\otimes f_\rho$ gives the inverse to the isomorphism $I: H^\bullet_{CR}(\B\times BG, c_1\boxtimes c_2)\to \oplus_{[\rho]\in \widehat{G}_{c_2}}H^\bullet_{CR}(B, c_1)$ as in \cite{sh-ta-ts}. Therefore Theorem \ref{thm:product} follows directly from Eq. (\ref{eqn:gw_c1c2_3}).

\section{Degree formula}\label{sec:degree} 
In this section, we explain the degree formula of the map 
\[
p: \Mbar_{g,n}(\Y, \beta; \mathfrak{l}_1,...,\mathfrak{l}_n)\to \Mbar_{g,n}(\B, \beta; \pi(\mathfrak{l}_1),...,\pi(\mathfrak{l}_n)). 
\]
We remark that though in this paper we only need the result of the degree formula for an abelian group  $G$, we have decided to develop the result for a general finite group $G$ for potential future applications. 

\subsection{Covering of an $n$-marked orbicurve}\label{subsec:covering}
Let $(\Sigma, \{(x_i, d_i) \}_{i=1}^n)$ be an orbifold Riemann surface of genus $g$ with $n$ marked orbifold points $x_1, \cdots, x_n$ whose orbifold structure at $x_i$ ($i=1,\cdots, n$) is a cyclic group of order $d_i$. In this subsection, we will describe an orbifold groupoid $\mathfrak{Q}_{\Sigma}$ presenting the orbifold Riemann surface $(\Sigma, \{(x_i, d_i)\})$. 

We start with a smooth Riemann surface $(\Sigma_0, \{x_i\}_{i=1}^n)$ of genus $g$ and $n$ marked smooth points $x_1, ..., x_n$. In the following, we construct a groupoid $\mathfrak{Q}_{\Sigma_0}$ presenting $(\Sigma_0, \{x_i\}_{i=1}^n)$ in several steps. 

\begin{enumerate}
\item \label{smooth:step1} For each $x_i$, choose a small disk $D_i$ center at $x_i$. Choose each $D_i$ small enough that $D_i$ intersects $D_j$ trivially when $i\ne j$. 
\item \label{smooth:step2} For each $D_i$, choose small disks $B_{i, J}$ $J=1, \cdots, m$ whose union covers the boundary $\partial D_{i}$. Without loss of generality, we can assume that the centers of $B_{i, J}$ are located in the counterclockwise order on the boundary $\partial D_i$, and  $B_{i, J}$ only intersects nontrivially with $B_{i, J-1}$ and $B_{i, J+1}$, (identify $B_{i, m+1}$ with $B_{i, 1}$), and $B_{i, J}$ intersects trivially with $B_{j, K}$ if $i\ne j$. 
\item \label{smooth:step3} Extend $\{B_{i, J}\}_{i, J}$ to an open cover $\{B_\alpha\}$ of the compliment $\Sigma_0-\cup_{i}D_i$.  As the dimension of $\Sigma_0-\cup_i D_i$ is 2, we can assume that there are no nontrivial quadriple intersections, and if $B_\alpha$ intersects $D_i$ nontrivially, $B_\alpha$ must be one of the $B_{i, J}$ chosen in the previous step. 
\item \label{smooth:step4} Define $M$ to be the disjoint union of all $B_\alpha$ and $D_i$, and $\mathfrak{Q}_{\Sigma_0}$ to be the disjoint union of $B_{\alpha}\cap B_{\beta}$, $B_{i, J}\cap D_i$, $D_i\cap B_{i,J}$, and $D_i$. Given a point $g$ in $\mathfrak{Q}_{\Sigma_0}$, the source (and target) map is defined as follows
\[
s(g):=\left\{\begin{array}{ll}
	g\in B_\alpha,& g\in B_\alpha\cap B_\beta,\\
	g\in B_{i, J},& g\in B_{i,J}\cap D_i,\\
	g\in D_i,& g\in D_i\cap B_{i, J},\\
	g\in D_i,& g\in D_i,
	\end{array}
	\right. 
t(g):=\left\{\begin{array}{ll}
	g\in B_\beta,& g\in B_\alpha\cap B_\beta,\\
	g\in D_i,& g\in B_{i,J}\cap D_i,\\
	g\in B_{i,J},& g\in D_i\cap B_{i, J},\\
	g\in D_i,& g\in D_i.
	\end{array}
	\right. 
\]
\end{enumerate}

The above construction of the groupoid is the \v Cech groupoid associated to the cover $\{B_\alpha, D_i\}$. We leave it to the reader to check that $\mathfrak{Q}_{\Sigma_0}\rightrightarrows M$ is a groupoid. 

Next, we consider an orbifold Riemann surface $(\Sigma, \{(x_i, d_i)\}) $ of genus $g$ with $n$ marked points $x_1, ..., x_n$ and $\integers_{d_i}:=\mathbb{Z}/ d_i\mathbb{Z}$-orbifold structure at $x_i$, $i=1,\cdots, n$. We construct the orbifold groupoid $\mathfrak{Q}_{\Sigma}$ in the following steps as a modification of the above construction. 
\begin{enumerate}
\item Follow Step (\ref{smooth:step1})-(\ref{smooth:step3}) to choose $D_i, B_{i, J}$ that cover the whole surface $\Sigma$. 
\item For each $x_i$, choose $\mathfrak{D}_i$ to be a unit disk equipped with the $\integers_{d_i}$ rotation action $\alpha_i$  on $\mathfrak{D}_i$ such that $\mathfrak{D}_i/\integers_{d_i}$ is homeomorphic to $D_i$. Define $X$ to be the disjoint union of all $B_\alpha$ and $\mathfrak{D}_i$. The arrow space $\calg_\Sigma$ consists of arrows of the following types.
	\begin{itemize}
		\item $g_{\alpha, \beta}\in B_{\alpha}\cap B_\beta$,
		\item $g_{i,M}=(x, M)\in \mathfrak{D}_i\times \integers/d_i \integers$ for $M\in \integers_{d_i}$,
		\item $g_{\alpha, i, M}=(x,M), g_{i, M, \alpha}=(x,M)\in \mathfrak{D}_i\times \integers_{d_i}$ such that $\pi(x)\in D_i\cap B_\alpha$. Use $[g_{\alpha, i,M}], [g_{i,M,\alpha}]$ to denote the image of $x$ in the quotient $D_i=\mathfrak{D}_i/\integers_{d_i}$. 				
	\end{itemize}
\item Define the groupoid structure on $\mathfrak{Q}_\Sigma\rightrightarrows X$ as follows. 
\[
\begin{array}{cc}
s(g):=\left\{\begin{array}{ll}
	g\in B_\alpha,& g=g_{\alpha, \beta},\\
	m \in \mathfrak{D}_i ,& g=g_{i,M},\\
	x\in \mathfrak{D}_i,& g=g_{i,M,\alpha},\\
	\pi(x) \in B_\alpha , & g=g_{ \alpha, i,M},\\
	\end{array}
	\right. ,
&
t(g):=\left\{\begin{array}{ll}
	g\in B_\beta,& g=g_{\alpha, \beta},\\
	\alpha_i^M(x)\in \mathfrak{D}_i,&g=g_{i,M},\\
	\pi(x)\in B_\alpha,& g=g_{i,M,\alpha},\\
	x\in \mathfrak{D}_i,& g=g_{\alpha, i,M}.
	\end{array}
	\right. 
\end{array}. 
\]
\end{enumerate}
With the above structures, it is not hard to figure out the definition of the product on $\calg_\Sigma$, which we leave to the reader. 

We point out that the small disks $D_i$ and $B_\alpha$ can be chosen to be sufficiently small such that the groupoid $\mathfrak{Q}_\Sigma$ and all its nerve spaces $\mathfrak{Q}_\Sigma^{\bullet}$ are disjoint unions of contractible open sets. Let $Z(G)$ be the center of $G$, and $B\mathfrak{Q}_\Sigma$ be the classifying space of the groupoid $\mathfrak{Q}_\Sigma$. A banded $G$-gerbe $\widetilde{\Sigma}$ over $\Sigma$ corresponds to a class in $H^2(B \mathfrak{Q}_\Sigma, Z(G))$. With our construction, such a class can be represented by $Z(G)$-valued continuous 2-cocycle $\tau$ on the groupoid $\mathfrak{Q}_\Sigma$. As $Z(G)$ is finite, such a cocycle is locally constant. 

\subsection{Holonomy of a banded gerbe} 

In this subsection, we introduce the notion of a holonomy for a banded $G$-gerbe $\Y_\Sigma$ over an oriented orbifold Riemann surface $\Sigma$.  

As is discussed in Sec. \ref{subsec:covering}, we represent $\Sigma$ by a proper \'etale groupoid $\mathfrak{Q}\rightrightarrows M$. Let $Z(G)$ be the center of the group $G$, and $\tau$ be a $Z(G)$-valued 2-cocycle on $\mathfrak{Q}$. We assume that the orbifold $\Y:=\Y_\Sigma$ is represented by the following groupoid 
\[
\mathfrak{H}:=G\times_\tau \mathfrak{Q}\rightrightarrows M,\ \text{with}\ (g_1, q_2)(g_2, q_2)=\big(g_1g_2\tau(q_1, q_2), q_1q_2\big). 
\]
Without loss of generality, we assume that $\tau$ is normalized, i.e. $\tau(e, q)=\tau(q,e)=1\in Z(G)$. 

The groupoids $\mathfrak{H}$ and $\mathfrak{Q}$ satisfy the following exact sequence,
\[
1\longrightarrow G\times M\rightrightarrows M \longrightarrow \mathfrak{H}\rightrightarrows M \longrightarrow \mathfrak{Q}\rightrightarrows M\longrightarrow 1. 
\]

Let $x$ be a point of $M$, and $\mathfrak{Q}_x$ be the stabilizer group of $\mathfrak{Q}$ at $x$, and $d_x$ be the order of $\mathfrak{Q}_x$. As $\mathfrak{Q}_x$ fixes $x$, $\mathfrak{Q}_x$ acts faithfully on $T_x M$ by rotation. Let $q$ be the unique element in $\mathfrak{Q}_x$ that acts on $T_x M$ by counter clockwise rotation with angle $2\pi/ d_x$. 
\begin{definition}\label{dfn:holonomy} Define the holonomy of $(\Y,\tau)$ at $x\in M$ to be 
\[
hol(\Y, \tau)(x):=\tau(e, q)\tau(q,q)\tau(q,q^2)\cdots \tau(q,q^{d_x-1})=\tau(q,q)\tau(q,q^2)\cdots \tau(q,q^{d_x-1})\in Z(G).
\] 
Given an $U(1)$-valued 2-cocycle on $\mathfrak{Q}$, we define the holonomy of $(\Y, \tau)$ at $x\in M$  with the exactly same formula
\[
hol(\Y, \tau)(x):=\tau(e, q)\tau(q,q)\tau(q,q^2)\cdots \tau(q,q^{d_x-1})=\tau(q,q)\tau(q,q^2)\cdots \tau(q,q^{d_x-1})\in U(1).
\]
\end{definition}

\begin{proposition}\label{prop:holonomy-dfn} 
Let $Z(G)^{d_x}$ be the subgroup of $Z(G)$ consisting of group elements whose orders are divisible by $d_x$, and $Z_x(G)$ be the quotient group $Z(G)/Z(G)^{d_x}$. 
\begin{enumerate} 
\item If there is a  $q\in \mathfrak{Q}$ with $s(q)=x$ and $t(q)=y$, $hol(\Y, \tau)(x)$ and $hol(\Y, \tau)(y)$ are equal in $Z_x(G)$. 
\item If $\tau$ and $\tau'$ in $Z^2(\mathfrak{Q}, Z(G))$ are different by $\delta \varphi$ where $\varphi$ is a $Z(G)$-valued 1-cochain on $\mathfrak{Q}$, $hol(\Y, \tau)(x)$ and $hol(\Y, \tau')(x)$ are equal in $Z_x(G)$.
\end{enumerate}
\end{proposition}
\begin{proof}

\noindent{(1)} Conjugation by $q$ gives an isomorphism between $\mathfrak{Q}_x$ and $\mathfrak{Q}_y$, i.e. $q^{-1} g q\in \mathfrak{Q}_y$ for any $g\in \mathfrak{Q}_x$. 
Consider $(1, q^{-1})(1,g)(1,q)$ in $\mathfrak{H}$. It is computed to be
\begin{equation}\label{eq:conjugation}
\Big(\tau\big(q^{-1},g\big)\tau\big(q^{-1}g, q\big), q^{-1}gq\Big)=\Big(\tau\big(q^{-1},g\big)\tau\big(q^{-1}g, q\big), 1\Big)\Big(1, q^{-1}gq\Big). 
\end{equation}
Therefore, $(1, q^{-1})(1,g)(1,q)(1, q^{-1})(1,g)(1,q)$ can be computed in two ways. If we compute the multiplication in the middle first, we have
\begin{equation}\label{eq:square}
\begin{split}
[(1, q^{-1})(1,g)(1,q)]^2&=(1,q^{-1})(1,g)\Big(\tau(q,q^{-1}),1\Big)(1,g)(1,q)\\ 
&=\Big(\tau(q,q^{-1}),1\Big)(1,q^{-1})(1,g)^2 (1,q).
\end{split}
\end{equation}
On the other hand, using Eq. (\ref{eq:conjugation}), we have
\begin{equation}\label{eq:middle}
[(1, q^{-1})(1,g)(1,q)]^2=\Big[\Big(\tau\big(q^{-1},g\big)\tau\big(q^{-1}g, q\big), 1\Big)\Big]^2 \Big(1, q^{-1}gq\Big)^2.
\end{equation}
Extending the above computation, $[(1, q^{-1})(1,g)(1,q)]^{d_x}$ can be computed in the following two ways. Generalizing Eq. (\ref{eq:square}), we
have 
\begin{eqnarray*}
[(1, q^{-1})(1,g)(1,q)]^{d_x}&=&\Big(\tau(q,q^{-1}),1\Big)^{d_x-1}(1,q^{-1})(1,g)^{d_x} (1,q)\\
&=&\Big(\tau(q,q^{-1}),1\Big)^{d_x-1}(1,q^{-1})(hol(\Y, \tau)(x),1)(1,q)\\
&=&\Big(\tau(q,q^{-1}),1\Big)^{d_x}(hol(\Y, \tau)(x),1)\\
&=&\Big(\tau(q,q^{-1})^{d_x}hol(\Y, \tau)(x),1\Big). 
\end{eqnarray*}
Generalizing Eq. (\ref{eq:middle}), we have
\begin{eqnarray*}
[(1, q^{-1})(1,g)(1,q)]^{d_x}&=&\Big[\Big(\tau\big(q^{-1},g\big)\tau\big(q^{-1}g, q\big), 1\Big)\Big]^{d_x} \Big(1, q^{-1}gq\Big)^{d_x}\\
&=&\Big(\tau\big(q^{-1},g\big)^{d_x}\tau\big(q^{-1}g, q\big)^{d_x}, 1\Big) \Big(hol(\Y, \tau)(y), 1\Big)\\
&=&\Big(\tau\big(q^{-1},g\big)^{d_x}\tau\big(q^{-1}g, q\big)^{d_x}hol(\Y, \tau)(y), 1\Big).
\end{eqnarray*}
The desired equation follows from the identification of the two different computations. \\

\noindent{(2)} As $\tau$ and $\tau'$ are different by a coboundary, we have 
\[
\tau(q_1, q_2)=\tau'(q_1, q_2)\varphi(q_1)\varphi(q_2)\varphi(q_1q_2)^{-1}. 
\]
When $q_1=q_2=q$, we have
\[
\tau(q,q)=\tau'(q, q)\varphi(q)\varphi(q)\varphi(q^2)^{-1}.
\]
When $q_1=q, q_2=q^k$, we have
\[
\tau(q,q^k)=\tau'(q, q^k)\varphi(q)\varphi(q^k)\varphi(q^{k+1})^{-1}.
\]
Multiplying $\tau(q,q^k)$ for $k=1, ..., d_x-1$, we have 
\[
hol(\Y, \tau)(x)=hol(\Y, \tau')\varphi(x)^{d_x},
\]
which implies the desired equation in $Z_x(G)$. 
\end{proof}

Proposition \ref{prop:holonomy-dfn} leads to the following definition. 
\begin{definition}\label{def:orbi-hol} Let $[x]$ be the associated point in the quotient $\Sigma$. The holonomy $Hol(\Y)([x])$ of $\Y_\Sigma$ at $[x]$ is defined to be the image of $hol(\Y, \tau)(x)$ in the group $Z_{[x]}(G):=Z_x(G)$. 
\end{definition}


\subsection{Liftings of  a banded orbicurve}
Let $(\Sigma, \{(x_i, d_i) \}_{i=1}^n)$ be an orbifold Riemann surface of genus $g$ with $n$ marked orbifold points $x_1, \cdots, x_n$ whose orbifold structure at $x_i$ ($i=1,\cdots, n$) is a cyclic group of order $d_i$. Let $f: \Sigma\to \B$ be a map such that the induced map on the stabilizer group of $\Sigma$ at every $x_i$ is injective for every $i$. Let $\Y$ be a banded $G$-gerbe on $\B$. In this subsection, we compute the number of orbifold Riemann surface $(\widetilde{\Sigma}, \{(\tilde{x}_i, \tilde{d}_i)\})_{i=1}^n$ together with a map $\tilde{f}:\widetilde{\Sigma}\to \Y$ and a covering map $p:\widetilde{\Sigma}\to\Sigma$ satisfying the following diagram 
\begin{equation}\label{eq:lifiting}
\xymatrix{
{\widetilde{\Sigma}}\ar[d]^{p}\ar[r]^{\tilde{f}} & \Y\ar[d]^{\pi}\\
\Sigma\ar[r]^{f} &\B.
} 
\end{equation}
Let $\Y_\Sigma:=f^*\Y$ be the pullback of $\Y$ along the map $f$. Then we are reduced to the following diagram
\[
\xymatrix{
&\Y_\Sigma \ar[d]^{\pi}\\
{\widetilde{\Sigma}}\ar[ur]^{\tilde{f}}\ar[r]^{p}&\Sigma.
}
\]
Furthermore, we pull $\Y_\Sigma$ back to $\widetilde{\Sigma}$ via the map $p: \widetilde{\Sigma}\to \Sigma$, and have the following diagram
\[
\xymatrix{
{p^*\Y_\Sigma} \ar[d] ^{\pi}\ar[r]^{\tilde{p}}&{f^*\Y}\ar[d]^\pi\\ 
{\widetilde{\Sigma}} \ar@/^/[u] ^{\tilde{f}}\ar[r]^{p}&\Sigma.
}
\]
With the above diagram, to count the number of $\widetilde{\Sigma}$ satisfying (\ref{eq:lifiting}), it is equivalent to count the number of coverings $\widetilde{\Sigma}$ together with sections $\tilde{f}: \widetilde{\Sigma}\to p^*\Y_\Sigma$. 

Let $\mathfrak{Q}_{\Sigma}$ be a proper \'etale groupoid presenting $(\Sigma, \{(x_i, d_i)\}_{i=1}^n)$ as constructed in Sec. \ref{subsec:covering}. A similar construction also defines a proper \'etale groupoid $\mathfrak{Q}_{\widetilde{\Sigma}}$ presenting the orbifold $\widetilde{\Sigma}$ together with a canonical map $p: \mathfrak{Q}_{\widetilde{\Sigma}}\to \mathfrak{Q}_{\Sigma}$.  Let $\tau$ be a $Z(G)$-valued 2-cocycle on the groupoid $\mathfrak{Q}_\Sigma$ representing the orbifold $\Y_\Sigma$.  $p^*(\tau)$ is a $Z(G)$-valued 2-cocycle on the groupoid $\mathfrak{Q}_{\widetilde{\Sigma}}$. The orbifolds $p^*\Y_\Sigma$ and $\Y_\Sigma$ can be represented by the groupoids
\[
\mathfrak{H}_{\widetilde{\Sigma}}:=G\times_{\tau} \mathfrak{Q}_{\widetilde{\Sigma}},\ \text{and}\ \mathfrak{H}_\Sigma:=G\times_\tau \mathfrak{Q}_{\Sigma}
\]
where the map $\pi: \mathfrak{H}_{\widetilde{\Sigma}}\to \mathfrak{Q}_{\widetilde{\Sigma}}$ is the forgetful map by $\pi(g, q)=q$.

As we have chosen the covering of $\widetilde{\Sigma}$ sufficiently fine, we can represent the morphism $\tilde{f}: \widetilde{\Sigma}\to p^*\Y_\Sigma$ by a groupoid morphism from $\mathfrak{Q}_{\widetilde{\Sigma}}$ to $\mathfrak{H}_{\widetilde{\Sigma}}$, which is called again $\tilde{f}: \mathfrak{Q}_{\widetilde{\Sigma}}\to \mathfrak{H}_{\widetilde{\Sigma}}$. As $\pi\circ \tilde{f}$ is the identity map, we can write the map $\tilde{f}: \mathfrak{Q}_{\widetilde{\Sigma}}\to \mathfrak{H}_{\widetilde{\Sigma}}$ in the following form $\tilde{f}(q)=(F(q), q)$, where $F: \mathfrak{Q}_{\widetilde{\Sigma}}\to G$ is a smooth map from $\mathfrak{Q}_{\widetilde{\Sigma}}$ to $G$. We call two morphisms $\tilde{f}_1$ and $\tilde{f}_2$ {\em equivalent} if there is a map $\varphi: X\to G$ such that 
\[
F_1(q)=\varphi(s(q))F_2(q)\varphi(t(q))^{-1},\ \forall q\in \mathfrak{Q}_{\widetilde{\Sigma}},
\]
where $F_1$ and $F_2$ from $\mathfrak{Q}_{\widetilde{\Sigma}}$ to $G$ are the components associated to $\tilde{f}_1$ and $\tilde{f}_2$. The following proposition follows directly from definitions. 

\begin{prop}\label{prop:F-conjugation} Given the orbifold Riemann surface $\widetilde{\Sigma}$ together with the map $p: \widetilde{\Sigma}\to \Sigma$, maps $\tilde{f}$  from $\widetilde{\Sigma}$ to $p^*\Y$ are in one-to-one correspondence with equivalent classes of maps from $\mathfrak{Q}_{\widetilde{\Sigma}}$ to $\mathfrak{H}_{\widetilde{\Sigma}}$. 
\end{prop}

Let $X^*$ be the disjoint union of $B_\alpha$ in Sec. \ref{subsec:covering}. The restrictions of $\mathfrak{Q}_\Sigma$ and $\mathfrak{Q}_{\widetilde{\Sigma}}$ to $X^*$ coincide and present the smooth punctured Riemann surface $\Sigma^*$. We denote this groupoid by $\mathfrak{Q}_{\Sigma^*}$, which is a subgroupoid of $\mathfrak{Q}_{\Sigma}$ and $\mathfrak{Q}_{\widetilde{\Sigma}}$. Let $\tau$ be the $Z(G)$-valued 2-cocycle on $\mathfrak{Q}_{\Sigma}$ that defines the gerbe $\Y_\Sigma$.  The restriction of $\tau$ to $\mathfrak{Q}_{\Sigma^*}$ is a $Z(G)$-valued 2-cocycle $\tau_{\Sigma^*}$ on $\mathfrak{Q}_{\Sigma^*}$. As the groupoid $\mathfrak{Q}_{\Sigma^*}$ is Morita equivalent to $\Sigma^*$ which is homotopy equivalent to its 1-skeleton, $H^2(B\mathfrak{Q}_{\Sigma^*}, Z(G))=H^2(\Sigma^*, Z(G))$ vanishes. Hence, $\tau_{\Sigma^*}$ is a coboundary $\delta \varphi$, where $\varphi$ is a $Z(G)$-valued 1-cochain on $\mathfrak{Q}_{\Sigma^*}$. As $\mathfrak{Q}_{\Sigma^*}$ is both open and closed in $\mathfrak{Q}_{\Sigma}$, $\varphi$ can be extended to a smooth $Z(G)$-valued 1-cochain $\psi$ on $\mathfrak{Q}$ by
\[
\psi(q):=\left\{\begin{array}{ll}\varphi(q),&q\in \mathfrak{Q}_{\Sigma^*},\\ 0,&\text{else}.\end{array}\right.
\] 
Without loss of generality, we will work below with the modified cocycle $\tau-\delta \psi$ on $\mathfrak{Q}_{\Sigma}$ whose restriction to $\mathfrak{Q}_{\Sigma^*}$ vanishes. We will still call this cocycle $\tau$, which pulls back to $p^*\tau$ a cocycle on $\mathfrak{Q}_{\widetilde{\Sigma}}$. 

Recall  that the groupoid $\mathfrak{Q}_{\Sigma}$ consists of arrows 
\[
g_{\alpha, \beta}, g_{i, M}, g_{\alpha, i,M}, g_{i, M, \alpha},
\]
satisfying the following multiplication rule
\begin{eqnarray*}
&g_{\alpha, \beta} g_{\beta, \gamma} =g_{\alpha, \gamma},\ &g_{\alpha, \beta}g_{\beta, i, M}=g_{\alpha, i, M},\\
&g_{\alpha,i, M}g_{i,M'} =g_{\alpha, i, M+M'},\ &g_{i, M}g_{i,M'}=g_{i, M+M'},\\
&g_{i,M}g_{i,M', \alpha} =g_{i, M+M', \alpha},\ &g_{i, M, \alpha} g_{\alpha, \beta}=g_{i, M, \beta}.
\end{eqnarray*}
For each $1\leq i\leq n$, $g_{i, 1}$ is of order $d_i$.  Similarly, the groupoid $\mathfrak{Q}_{\widetilde{\Sigma}}$ consists of arrows 
\[
g_{\alpha, \beta}, \tilde{g}_{i, M_i+N_id_i}, \tilde{g}_{\alpha, i, M_i+N_id_i}, \tilde{g}_{i, M_i+N_id_i, \alpha}, 
\]
for $i=1,\cdots, n$, and $M_i=1, \cdots, d_i-1 $, and $N_i=1, \cdots, m_i-1$, and $g_{i,1}$ is of order $m_id_i$.  The above arrows for $\mathfrak{Q}_{\widetilde{\Sigma}}$ satisfy the similar relations as those of $\mathfrak{Q}_{\Sigma}$. The map $p: \widetilde{\Sigma}\to \Sigma$ can be realized by a groupoid morphism $p:\mathfrak{Q}_{\widetilde{\Sigma}}\to \mathfrak{Q}_{\Sigma}$, i.e. 
\[
p(g_{\alpha, \beta})=g_{\alpha, \beta},\ p(\tilde{g}_{i, M_i+N_i d_i})=g_{i, M_i},\ p(\tilde{g}_{\alpha, i, M_i+N_id_i})=g_{\alpha, i, M_i},\ p(\tilde{g}_{i, M_i+N_id_i, \alpha})=g_{i, M_i, \alpha}. 
\]

Recall that $\tilde{f}:\mathfrak{Q}_{\widetilde{\Sigma}}\to \mathfrak{H}_{\widetilde{\Sigma}}$ is a groupoid morphism. Accordingly, we have the following equation
\[
F(q_1)F(q_2)\tau(q_1, q_2)=F(q_1q_2), 
\]
for $q_1, q_2\in \mathfrak{Q}_{\widetilde{\Sigma}}$.  In term of the arrows $g_{\alpha, \beta}, \tilde{g}_{i, M}, \tilde{g}_{\alpha, i,M}, \tilde{g}_{i, M, \alpha}$, we have the following equations
\begin{equation}\label{eq:F}
\begin{array}{ll}
F(g_{\alpha, \beta}) F(g_{\beta, \gamma}) =F(g_{\alpha, \gamma}),\ &F(g_{\alpha, \beta})F(\tilde{g}_{\beta, i, M})\tau(g_{\alpha, \beta}, \tilde{g}_{\beta, i, M})=F(\tilde{g}_{\alpha, i, M}),\\
F(\tilde{g}_{\alpha,i, M})F(\tilde{g}_{i,M'} )\tau(\tilde{g}_{\alpha, i, M}, \tilde{g}_{i, M'})=F(\tilde{g}_{\alpha, i, M+M'}),\ &F(\tilde{g}_{i, M})F(\tilde{g}_{i,M'})\tau(\tilde{g}_{i, M}, \tilde{g}_{i, M'})=F(\tilde{g}_{i, M+M'}),\\
F(\tilde{g}_{i,M})F(\tilde{g}_{i,M', \alpha})\tau(\tilde{g}_{i, M}, \tilde{g}_{i, M', \alpha}) =F(\tilde{g}_{i, M+M', \alpha}),\ &F(\tilde{g}_{i, M, \alpha})F(\tilde{g}_{\alpha, \beta})\tau(\tilde{g}_{i, M, \alpha}, \tilde{g}_{\alpha, \beta})=F(\tilde{g}_{i, M, \beta}). 
\end{array}
\end{equation}
As $F$ can be changed by a coboundary, without loss of generality, in the following computation, we always assume that $F$ maps units in $\mathfrak{Q}_{\widetilde{\Sigma}}$  to the identity element in $G$. 

\begin{proposition}\label{prop:sigma*}The restriction of the map $F$ to $\mathfrak{Q}_{\Sigma^*}$ is groupoid morphism from $\mathfrak{Q}_{\Sigma^*}$ to $G$. This defines a map from $\Sigma^*$ to the classifying space $BG$ of principal $G$-bundles.  Equivalent classes of such maps correspond to isomorphism classes of principal $G$-bundles over $\Sigma^*$. 
\end{proposition}
\begin{proof}This is a corollary from the first equation in Eqs. (\ref{eq:F}), i.e.
\[
F(g_{\alpha, \beta}) F(g_{\beta, \gamma}) =F(g_{\alpha, \gamma}). 
\]
The remaining part of the proposition follows directly from the definition of principal $G$ bundle. 
\end{proof}

In the following, we describe the topological conditions that the morphism $\tilde{f}$ satisfies. Choose a base point $x_0$ in $\Sigma$. We first make a careful choice of generators of the fundamental group of the punctured surface $\Sigma^*:=\Sigma-\{x_1, \cdots, x_n\}$. 

For each $D_i$ ($i=1,...,n$), choose a point $w_i$ on the boundary of $D_i$, and a corresponding point $z_i$ on the boundary of $\mathfrak{D}_i$ which under the quotient map is mapped to $w_i$.  For each $i$ choose a path $l_i$ in $\Sigma^*$ connecting $x_0$ to $y_0$, and call the loop starting from $x_0$ followed by $l_i$, the boundary of $D_i$ (oriented counter clockwise), and the inverse of $l_i$ by $L_i$.  Recall that in the cover introduced in Sec. \ref{subsec:covering}, we have chosen $B_{i, 1}, \cdots, B_{i, J}$ to cover the boundary of  $D_i$. We extend this cover by $C_{i,0}, C_{i, 1}, \cdots, C_{i, K}$ to cover the path $l_i$ such that $C_{i,K}=B_{i,1}$. We assume that $C_{1,0}=C_{2,0}=...=C_{n,0}$, and the disks are chosen to be sufficiently small that there are no nontrivial triple intersections. For each $i$, we order $\{ C_{i, k}\} $  so that $C_{i, k}$ only intersects $C_{i, k-1}$ and $C_{i, k+1}$ nontrivially.  We denote an arrow in $\mathfrak{Q}_{\Sigma^*}$ starting from $C_{i,s}$ to $C_{i,s+1}$ and $C_{i,s-1}$ by $g_{i, s,s+1}$ and $g_{i, s,s-1}$, and denote an arrow in $\mathfrak{Q}_{\Sigma^*}$ starting from $B_{i, t}$ to $B_{i, t+1}$ by $h_{i, t,t+1}$.

\begin{definition}\label{dfn:F-fundamental} For $j=1, \cdots, n$, define $F(L_j)$ to be 
\[
\begin{split}
F(L_j):=&F\big(g_{(j,0), (j,1)}\big) F\big(g_{(j, 1), (j, 2)}\big)\cdots F\big(g_{(j,K-1), (j, K)}\big)\\
&F\big(h_{j,1,2}\big)  F\big(h_{j, 2,3}\big)\cdots F\big(h_{j, J, 1}\big)F\big(g_{(j,K), (j, K-1)}\big)\cdots F\big(g_{(j,2),(j,1)}\big)F\big(g_{(j,1),(j,0)}\big).
\end{split}
\]
Define $F(l_j)$ and $F(l_j^{-1})$ to be
\[
\begin{split}
F(l_j):=&F\big(g_{(j,0), (j,1)}\big) F\big(g_{(j, 1), (j, 2)}\big)\cdots F\big(g_{(j,K-1), (j, K)}\big),\\ 
F(l_j^{-1}):=&F\big(g_{(j,K), (j, K-1)}\big)\cdots F\big(g_{(j,2),(j,1)}\big)F\big(g_{(j,1),(j,0)}\big).
\end{split}
\]
\end{definition}
We remark that as $G$ is discrete and $F$ is smooth, $F$ is locally constant.
As we have assumed that $F$ maps units of $\mathfrak{Q}_{\widetilde{\Sigma}}$ to identity of $G$, we have that 
\[
F(l_j)F(l_j^{-1})=1,\qquad F(l_j^{-1})=F(l_j)^{-1}.
\]

As Definition \ref{dfn:F-fundamental}, we fix $x_i, y_i$, $i=1,\cdots, g$ to be loops in the smooth Riemann surface $\Sigma$ of genus $g$ based at $x_0$, and cover $x_i$ and $y_i$ by $A_{i, p}$, $B_{i, q}$ counter clockwise. Let $x_{i, s,s+1}$ and $y_{i, t,t+1}$ be the arrows in $\mathfrak{Q}_{\Sigma^*}$ starting from $A_{i, s}$ and $B_{i, t}$ and ending in $A_{i, s+1}$ and $B_{i,t+1}$, $s=1,..., J,\ t=1,..., K$. 
\begin{definition}
\label{dfn:F-xy} For $i=1,\cdots, g$, define $F(x_i)$ and $F(y_i)$ to be
\[
\begin{split}
F(x_i):=F(x_{i,1,2})&\cdots F(x_{i,J,1 }),\qquad F(y_i):=F(y_{i,1,2})\cdots F(y_{i,J,1 }).
\end{split}
\]
\end{definition}

Assume that $x_i, y_i$ generate the fundamental group $\pi_1(\Sigma, x_0)$ satisfying the following relation
\[
\langle x_i, y_i | [x_1, y_1][x_2, y_2]\cdots [x_g, y_g]=1 \rangle,
\]
where the commutator $[x,y]$ is defined to be $xyx^{-1}y^{-1}$.  Similarly, the fundamental group $\pi_1(\Sigma^*, x_0)$ is generated by $x_i, y_i, L_j$, $i=1, \cdots, g$, $j=1, \cdots, n$, satisfying
\begin{equation}\label{eq:fundamental}
 [x_1, y_1][x_2, y_2]\cdots [x_g, y_g]L_1\cdots L_n=1. 
\end{equation}

Following Eq. (\ref{eq:F}), as $g_{(j,s), (j,s+1)}$, $h_{j,t,t+1}$, $x_{i,s,s+1}$, and $y_{i,t,t+1}$ all belong to $\mathfrak{Q}_{\Sigma^*}$ on which the cocycle $\tau_{\Sigma^*}$ vanishes, $F$ is a groupoid morphism from $\mathfrak{Q}_{\Sigma^*}$ to $G$. Globally, as $F$ is locally constant, $F$ defines a flat principal $G$-bundle on $\Sigma^*$. And equivalent classes of morphisms $F: \mathfrak{Q}_{\Sigma^*}\to G$  corresponding to isomorphism classes of flat $G$-bundles on $\Sigma^*$. Given $x_0\in \Sigma^*$, isomorphism classes of flat $G$-bundles on $\Sigma^*$ are in one-to-one correspondence with elements of the set $\operatorname{Hom}(\pi(\Sigma^*, x_0), G)/G$, where the group $G$ acts on $\operatorname{Hom}(\pi(\Sigma^*, x_0), G)$ by conjugation. More concretely, the above map $F$ induces a map $\widetilde{F}\in \operatorname{Hom}(\pi(\Sigma^*, x_0), G)/G$ that maps $L_j$ to $F(L_j)$, $x_i$ to $F(x_i)$, and $y_i$ to $F(y_i)$, where are introduced in Definition \ref{dfn:F-fundamental} and \ref{dfn:F-xy}. With the relation (\ref{eq:fundamental}), the corresponding $F(x_i)$, $F(y_i)$, and $F(L_j)$ satisfy the following equation,
\begin{equation}\label{eq:F-relation}
[F(x_1), F(y_1)][F(x_2), F(y_2)]\cdots [F(x_g), F(y_g)]F(L_1)\cdots F(L_n)=1. 
\end{equation}

\begin{lemma}\label{lem:F(L)} For $j=1, \cdots,n$, 
\[
\begin{split}
F(L_j)=\Big[\tau\big(\tilde{g}_{1,j,M}, \tilde{g}_{j,1}\big)&\tau\big(h_{j,1,2}, \tilde{g}_{2, j, M}\big)\cdots \tau\big(h_{j,J,1}, \tilde{g}_{1, j, M+1}\big)\Big]^{-1}\\
&F(l_j) F\big(\tilde{g}_{1, j, M+1}\big)F\big(\tilde{g}_{j,1}\big)^{-1} \Big(F(l_j)F\big(\tilde{g}_{1,j,M+1}\big)\Big)^{-1}. 
\end{split}
\]

\end{lemma}
\begin{proof} Following Definition \ref{dfn:F-fundamental} of $F(L_j)$ and the property that $F$ is a groupoid morphism from $\mathfrak{Q}_{\Sigma^*}$ to $G$, we can write 
\[
F(L_j)=F(l_i)F\big(h_{j,1,2}\big)  F\big(h_{j, 2,3}\big)\cdots F\big(h_{j, J, 1}\big)F(l_i)^{-1}.
\]
Recall that from Eq. (\ref{eq:F}), we have
\[
F\big(h_{j,s,s+1}\big) F\big(\tilde{g}_{s+1, j,M}\big)\tau\big(h_{j,s,s+1}, \tilde{g}_{s+1, j, M}\big)=F\big(\tilde{g}_{s,j, M}\big).
\]
Hence, as $\tau$ takes value in the center of $G$, we can write 
\[
F\big(h_{j,s,s+1}\big)=\tau\big(h_{j,s,s+1}, \tilde{g}_{s+1, j, M}\big)^{-1}F\big(\tilde{g}_{s,j, M}\big)F\big(\tilde{g}_{s+1, j,M}\big)^{-1}. 
\]
Multiplying the above equations, we conclude the following equation
\[
\begin{split}
&F\big(h_{j,1,2}\big)  F\big(h_{j, 2,3}\big)\cdots F\big(h_{j, J, 1}\big)\\
=&\Big[\tau\big(h_{j,1,2}, \tilde{g}_{2, j, M}\big)\cdots \tau\big(h_{j,J,1}, \tilde{g}_{1, j, M+1}\big)\Big]^{-1} F\big(\tilde{g}_{1,j, M}\big)F\big(\tilde{g}_{1,j, M+1}\big)^{-1}.
\end{split}
\]

To deal with the last two terms in the above equation, we recall from Eq. (\ref{eq:F}) that
\[
F\big(\tilde{g}_{1,j,M}\big)F\big(\tilde{g}_{j,1}\big)\tau(\tilde{g}_{1,j,M}, \tilde{g}_{j,1})=F\big(\tilde{g}_{1, j, M+1}\big),
\]
and 
\[
F\big(\tilde{g}_{1,j, M}\big)F\big(\tilde{g}_{1,j, M+1}\big)^{-1}=\tau(\tilde{g}_{1,j,M}, \tilde{g}_{j,1})^{-1}F\big(\tilde{g}_{1, j, M+1}\big)F\big(\tilde{g}_{j,1}\big)^{-1} F\big(\tilde{g}_{1,j,M+1}\big)^{-1}. 
\]
In summary, we have the following expression for $F\big(h_{j,1,2}\big)  F\big(h_{j, 2,3}\big)\cdots F\big(h_{j, J, 1}\big)$
\[
\Big[\tau\big(\tilde{g}_{1,j,M}, \tilde{g}_{j,1}\big)\tau\big(h_{j,1,2}, \tilde{g}_{2, j, M}\big)\cdots \tau\big(h_{j,J,1}, \tilde{g}_{1, j, M+1}\big)\Big]^{-1} F\big(\tilde{g}_{1, j, M+1}\big)F\big(\tilde{g}_{j,1}\big)^{-1} F\big(\tilde{g}_{1,j,M+1}\big)^{-1}. 
\]
We can write $F(L_j)$ as 
\[
\begin{split}
\Big[\tau\big(\tilde{g}_{1,j,M}, \tilde{g}_{j,1}\big)&\tau\big(h_{j,1,2}, \tilde{g}_{2, j, M}\big)\cdots \tau\big(h_{j,J,1}, \tilde{g}_{1, j, M+1}\big)\Big]^{-1}\\
&F(l_j) F\big(\tilde{g}_{1, j, M+1}\big)F\big(\tilde{g}_{j,1}\big)^{-1} \Big(F(l_j)F\big(\tilde{g}_{1,j,M+1}\big)\Big)^{-1}. 
\end{split}
\]
\end{proof}

Lemma \ref{lem:F(L)} suggests the following definition. 
\begin{definition}\label{dfn:holonomy-L}
The holonomy of the $G$-gerbe $(\Y_{\Sigma},\tau)$ around the loop $L_j$ at $z_j$ is defined to be
\[
hol(\Y_\Sigma, \tau)(L_j, z_j):=\tau\big(\tilde{g}_{1,j,M}, \tilde{g}_{j,1}\big)\tau\big(h_{j,1,2}, \tilde{g}_{2, j, M}\big)\cdots \tau\big(h_{j,J,1}, \tilde{g}_{1, j, M+1}\big).
\]

\end{definition}
\begin{remark} We remark that Definition \ref{dfn:holonomy-L} also works for a general Lie group. In particular, we will use this definition later for $G=U(1)$.  From a 2-cocycle, we can find a natural connection $\nabla^\tau$ on $\Y_\Sigma$. Our definition agrees with the holonomy defined by $\nabla^\tau$ in \cite{lu} and \cite{pry}. 
\end{remark}

The following lemma asserts that the above definition is independent of the variable $M$.
\begin{lemma}\label{lem:hol-ind-M} For any $M$, 
\[
\begin{split}
&\tau\big(\tilde{g}_{1,j,M}, \tilde{g}_{j,1}\big)\tau\big(h_{j,1,2}, \tilde{g}_{2, j, M}\big)\cdots \tau\big(h_{j,J,1}, \tilde{g}_{1, j, M+1}\big)\\
=&\tau\big(\tilde{g}_{1,j,M+1}, \tilde{g}_{j,1}\big)\tau\big(h_{j,1,2}, \tilde{g}_{2, j, M+1}\big)\cdots \tau\big(h_{j,J,1}, \tilde{g}_{1, j, M+2}\big).
\end{split}
\]
\end{lemma}
\begin{proof}
In $\mathfrak{Q}_{\Sigma}$, the arrows $h_{j, s, s+1}$, $\tilde{g}_{s+1, j, M}$, $\tilde{g}_{j,1}$ are composable. By the cocycle condition for $\tau$, we have
\[
\tau(h_{j, s,s+1}, \tilde{g}_{s+1, j, M})\tau(\tilde{g}_{s, j, M}, \tilde{g}_{j,1})=\tau(h_{j,s,s+1}, \tilde{g}_{s+1, j, M+1})\tau(\tilde{g}_{s+1, j, M}, \tilde{g}_{j,1}). 
\]
Moving the second term on the left side to the right, we obtain
\[
\tau(h_{j, s,s+1}, \tilde{g}_{s+1, j, M})=\tau(h_{j,s,s+1}, \tilde{g}_{s+1, j, M+1})\tau(\tilde{g}_{s+1, j, M}, \tilde{g}_{j,1})\tau(\tilde{g}_{s, j, M}, \tilde{g}_{j,1})^{-1}. 
\]
Multiplying the above equations for $s=1,...,J$, we 
have
\[
\begin{split}
&\tau\big(h_{j,1,2}, \tilde{g}_{2, j, M}\big)\cdots \tau\big(h_{j,J,1}, \tilde{g}_{1, j, M+1}\big)\\
=&\tau\big(h_{j,1,2}, \tilde{g}_{2, j, M+1}\big)\cdots \tau\big(h_{j,J,1}, \tilde{g}_{1, j, M+2}\big)\tau\big(\tilde{g}_{1,j, M+1}, \tilde{g}_{j,1}\big)\tau(\tilde{g}_{1, j, M}, \tilde{g}_{j,1})^{-1}.
\end{split}
\]
Moving the last term on the right side to the left, we have obtained the desired identity. 
\end{proof}

\begin{theorem}\label{thm:degree} 
Given a banded $G$-gerbe $\Y_\Sigma$ over the orbifold Riemann surface $\Sigma$, the number  orbifold Riemann surfaces $\widetilde{\Sigma}$ together with the structures $p:\widetilde{\Sigma}\to \Sigma$ and $\tilde{f}: \widetilde{\Sigma}\to p^*\Y_\Sigma$ satisfying the following commutative diagram, 
\[
\xymatrix{
{p^*\Y_\Sigma} \ar[d] ^{\pi}\ar[r]^{\tilde{p}}&{\Y_\Sigma}\ar[d]^\pi\\ 
{\widetilde{\Sigma}} \ar@/^/[u] ^{\tilde{f}}\ar[r]^{p}&\Sigma
}
\]
is equal to the cardinality of the quotient of the following set 
\[
\begin{array}{ll}
&\chi(\tau, g, d_1,\cdots, d_n):=\\
&\left\{
\begin{aligned}
&\big(F(x_1), F(y_1),\cdots,F(x_g), F(y_g), F(\tilde{g}_{j,1}), \cdots, F(\tilde{g}_{j, n})\big)\in G^{\times (2g+n)}:\\
&[F(x_1), F(y_1)]\cdots [F(x_g), F(y_g)]\\
&\qquad \qquad \qquad =hol(\Y_\Sigma, \tau)(L_1, z_1)\cdots hol(\Y_\Sigma, \tau)(L_n, z_n) F(\tilde{g}_{n,1})\cdots F(\tilde{g}_{1,1}),\\
&s.t.\ F(\tilde{g}_{1,1})^{m_1d_1} hol(\Y_\Sigma, \tau)(x_1)^{m_1}=1,\cdots,\ F(\tilde{g}_{n,1})^{m_nd_n}hol(\Y_\Sigma, \tau)(x_n)^{m_n}=1
\end{aligned}
\right\}.
\end{array}
\]
by the diagonal conjugacy action of $G$. In the above description, $d_i$ is the order of the stabilizer group of $\Sigma$ at $x_i$ and $m_id_i$ is the order of the stabilizer group of $\widetilde{\Sigma}$ at $x_i$, $i=1,\cdots,n$. 
\end{theorem}

\begin{remark}
Suppose $G$ is abelian. It is easy to see that the defining equation of $\chi(\tau, g, d_1,\cdots, d_n)$ is unchanged when two marked points with opposite orbifold structures are identified to form a node. Thus by induction on the number of nodes, for a {\em nodal} orbifold Riemann surface $\Sigma$, if a lifting $\Sigma\overset{p}{\longleftarrow} \widetilde{\Sigma}\to p^*\Y_\Sigma$ exists, then the defining equation of $\chi(\tau, g, d_1,\cdots, d_n)$ also holds.
\end{remark}

\begin{proof}
Lemma \ref{lem:F(L)} and Definition \ref{dfn:holonomy-L} give the following expression of $F(L_j)$
\[
F(L_j)=hol(\Y_\Sigma, \tau)(L_j, z_j)^{-1}F(l_j) F\big(\tilde{g}_{1, j, M+1}\big)F\big(\tilde{g}_{j,1}\big)^{-1} \Big(F(l_j)F\big(\tilde{g}_{1,j,M+1}\big)\Big)^{-1}.
\]
Around $x_j$, the equivalence relation in the definition of the map $F: \mathfrak{Q}_{\widetilde{\Sigma}}\to G$ allows to change $F(\tilde{g}_{i,M})$, $F(\tilde{g}_{\alpha, i, M})$, and $\tilde{g}_{i, M, \alpha}$ by
\[
F(\tilde{g}_{i,M})\cong g F(\tilde{g}_{i,M})g^{-1},\ F(\tilde{g}_{\alpha, i, M})\cong F(\tilde{g}_{\alpha, i, M})g^{-1},\ F(\tilde{g}_{i, M, \alpha})\cong g F(\tilde{g}_{i,M, \alpha}),
\]
for $g\in G$. As a corollary, we can choose $F(\tilde{g}_{1,j,M+1})$ to be equal to $F(l_j)^{-1}$, and conclude that
\[
F(L_j)=hol(\Y_\Sigma, \tau)(L_j, z_j)^{-1}F(\tilde{g}_{j, 1})^{-1}. 
\]
Inserting the above expression of $F(L_j)$ into Eq. (\ref{eq:F-relation}), we conclude the following equation of 
\[
[F(x_1), F(y_1)]\cdots [F(x_g), F(y_g)]=hol(\Y_\Sigma, \tau)(L_1, z_1)\cdots hol(\Y_\Sigma, \tau)(L_n, z_n) F(\tilde{g}_{n,1})\cdots F(\tilde{g}_{1,1}). 
\]

Using Eq. (\ref{eq:F}), we compute $F(\tilde{g}_{i,1})^2$ to be 
\[
\tau(\tilde{g}_{i,1},\tilde{g}_{i,1} )^{-1}F(\tilde{g}_{i, 2}). 
\]
Repeat the above computation, we get by using the fact $\tilde{g}_{i,1}^{m_id_i}=1$, 
\[
\begin{split}
F(\tilde{g}_{i,1})^{m_id_i}&=\tau(\tilde{g}_{i,1}, \tilde{g}_{i,1})^{-1}\cdots \tau(\tilde{g}_{i, m_id_i-1}, \tilde{g}_{i, 1})^{-1}=\big(\tau(\tilde{g}_{i,1}, \tilde{g}_{i,1})\cdots \tau(\tilde{g}_{i, d_i-1}, \tilde{g}_{i,1})\big)^{-m_i}\\
&=hol(\Y_\Sigma, \tau)(x_i)^{-m_i}. 
\end{split}
\]
Therefore, we conclude that the map $F: \mathfrak{Q}_{\widetilde{\Sigma}}\to G$ associated to the map $\tilde{f}: \widetilde{\Sigma}\to p^*\Y_\Sigma$, up to an equivalence in the definition of $F(\tilde{g}_{1,j,1})$, satisfies the conditions in the definition of $\chi(\tau, g,d_1,\cdots, d_n)$. It is not hard to check that if $F_1$ and $F_2$ are equivalent, then the corresponding  $F(x_i)$, $F(y_i)$, and $F(\tilde{g}_{j, 1})$ are changed by conjugation with respect to a same element $g$ in $G$. Therefore, we have constructed a natural map $H$ from the set of equivalent classes of maps of $F:\mathfrak{Q}_{\widetilde{\Sigma}}\to G$ to the set 
\[
\chi(\tau, g,d_1,\cdots, d_n)/G.
\]

Conversely, let $\big(F(x_1),F(y_1),\cdots, F(x_g), F(y_g), F(\tilde{g}_{1, 1}), \cdots, F(\tilde{g}_{n,1})\big)$ be a set of choices satisfying the conditions in $\chi(\tau, g, d_1,\cdots, d_n)$. Set 
\[
F(L_j):=hol(\Y_\Sigma, \tau)(L_j, z_j)^{-1} F(\tilde{g}_{j,1})^{-1}. 
\]
Then the data $\big(F(x_1), F(y_1), \cdots, F(x_g), F(y_g), F(L_1), \cdots, F(L_n)\big)$ satisfy the following equation, 
\[
[F(x_1), F(y_1)]\cdots [F(x_g), F(y_g)]F(L_1)\cdots F(L_n)=1. 
\]
Over the punctured smooth surface $\Sigma^*$, the choices of $F(x_i)$, $F(y_i)$, and $F(L_j)$ uniquely determine an isomorphism class of principal $G$-bundles over $\Sigma^*$, and by Proposition \ref{prop:sigma*} also uniquely determine a map $\tilde{f}$ from $\Sigma^*$ to $\Y_{\Sigma^*}:=(\Y_\Sigma )|_{\Sigma^*} $ up to equivalence. Using the definition of $F$ on $\Sigma^*$, we define $F(l_j)$, $j=1,\cdots, n$, following Definition \ref{dfn:F-fundamental}.  Set $F(\tilde{g}_{1,j,1})$ to be $F(l_j)^{-1}$. Following Eq. (\ref{eq:F}), we set 
\[
\begin{split}
 F\big(\tilde{g}_{s+1, j,M}\big):=&F\big(h_{j,s,s+1}\big)^{-1}F\big(\tilde{g}_{s,j, M}\big)\tau\big(h_{j,s,s+1}, \tilde{g}_{s+1, j, M}\big)^{-1},\\
  F\big(\tilde{g}_{j,M,s+1}\big):=&F\big(\tilde{g}_{j, M,s}\big)\tau\big(\tilde{g}_{j, M,s+1},h_{j,s+1,s}\big)^{-1}F\big(h_{j,s+1,s}\big)^{-1}.
\end{split}
\]
It is straightforward to check that the above $F\big(\tilde{g}_{s+1, j,M}\big)$, $ F\big(\tilde{g}_{j,M,s+1}\big)$, and $F\big(\tilde{g}_{j,1}\big)$ define an extension of $F$ to a map $\widetilde{F}:\mathfrak{Q}_{\widetilde{\Sigma}}\to G$ that induces a groupoid morphism $\tilde{f}: \mathfrak{Q}_{\widetilde{\Sigma}}\to \mathfrak{H}_{\widetilde{\Sigma}}$.  
It is not hard to check that if the choices of $(F(x_i), F(y_i), F(\tilde{g}_{j,1}))$ are changed by a conjugation of $g$ in $G$, the corresponding map $\tilde{f}$ is changed under equivalence. Therefore, we have constructed  a map 
$E$ from $\chi(\tau, g, d_1,\cdots, d_n)/G$ to the set of diffeomorphism classes of orbifold Riemann surfaces $\widetilde{\Sigma}$ with the desired properties in the statement of theorem. 

With the definitions of $H$ and $E$, it is not difficult to check that they are inverse to each other. Therefore we can compute the number of orbifold Riemann surfaces $\widetilde{\Sigma}$ with the covering map $p:\widetilde{\Sigma}\to \Sigma$ and the lifting map $\tilde{f}: \widetilde{\Sigma}\to p^*\Y_\Sigma$ by the cardinality of the set $\chi(\tau, g, d_1,\cdots, d_n)/G.$
\end{proof} 

\subsection{A combinatorial degree formula} 
Fix a stratum of $\Mbar_{g,n}(\Y,\beta)$ by the fixed point data $\LL=(\mathfrak{l}_1,\cdots, \mathfrak{l}_n)$. Let $\pi(\LL)$ be the corresponding fixed point data for $\Mbar_{g,n}(\B,\beta)$. We provide in this subsection a formula computing the degree of the map 
\[
\mathcal{M}_{g,n}(\Y,\beta, \LL)\longrightarrow \mathcal{M}_{g,n}(\B,\beta, \pi(\LL)),
\]
where $\M\subset \Mbar$ parametrizes stable maps with nonsingular domains. The space $\Mbar_{g,n}(\Y, \beta)$ parametrizes orbifold stable maps {\em with sections to all marked gerbes}, see \cite{agv1}. Following \cite{av} (with a small change of notations), denote by $$\overline{\mathcal{K}}_{g,n}(\Y, \beta)$$ the moduli space of $n$-pointed genus $g$ degree $\beta$ stable maps to $\Y$ {\em without} requiring that sections to marked gerbes exist. Forgetting sections gives a map 
$$\Mbar_{g,n}(\Y, \beta)\to \overline{\mathcal{K}}_{g,n}(\Y, \beta),$$ 
which exhibits $\Mbar_{g,n}(\Y, \beta)$ as the fiber product of the $n$ universal marked gerbes over $\overline{\mathcal{K}}_{g,n}(\Y, \beta)$. Restricting to $\Mbar_{g,n}(\Y,\beta, \LL)$ gives a map 
$$\Mbar_{g,n}(\Y, \beta, \LL)\to \overline{\mathcal{K}}_{g,n}(\Y, \beta, \LL).$$ 
This map has degree $1/w(\LL)$, where $w(\LL)$ is the product of orders of orbifold structures at marked points. Similar discussions apply to $\Mbar_{g,n}(\B, \beta)$, giving a map 
$$\Mbar_{g,n}(\B, \beta, \pi(\LL))\to \overline{\mathcal{K}}_{g,n}(\B, \beta, \pi(\LL))$$
 of degree $1/w(\pi(\LL))$. These maps fit into a commutative diagram
\begin{equation*}
\xymatrix{
\Mbar_{g,n}(\Y, \beta, \LL)\ar[r]\ar[d] & \Mbar_{g,n}(\B, \beta, \pi(\LL))\ar[d]\\
\overline{\mathcal{K}}_{g,n}(\Y, \beta, \LL)\ar[r] & \overline{\mathcal{K}}_{g,n}(\B, \beta, \pi(\LL)).
}
\end{equation*}
The degree of the top map is the degree of the bottom map times the factor $w(\pi(\LL))/w(\LL)$. In what follows we compute the (local) degree of $$p_{\LL}: \mathcal{K}_{g,n}(\Y, \beta, \LL)\to \mathcal{K}_{g,n}(\B, \beta, \pi(\LL)),$$ where $\mathcal{K}\subset \overline{\mathcal{K}}$ parametrizes stable maps with nonsingular domains. 

Let $\Y_y$ be the stabilizer group at $y\in \Y$, and $\B_{\pi(y)}$ be the stabilizer group of  $\pi(y)\in \B$. $\Y_y$ is an extension of $\B_{\pi(y)}$ by $G$ via the $
Z(G)$-valued 2-cocycle $\tau|_{\pi(y)}$, the restriction of $\tau$ on $\B_{\pi(y)}$. Let $\mathfrak{l}$ be a conjugacy class of $\Y_y$ and $\pi(\mathfrak{l})$ be the corresponding conjugacy class of $\B_{\pi(y)}$. As $\Y_y$ is a central extension of $\B_{\pi(y)}$, we have the following description of $\mathfrak{l}$.
\begin{lemma}\label{lem:conjugacy} There are conjugacy classes $(g_1), \cdots, (g_k)$ of the group $G$ such that the conjugacy class $\mathfrak{l}$ is of the following form 
\[
\{(g, q)| q\in \pi(\mathfrak{l}), g\in (g_i), i=1,\cdots, k.\}
\]
\end{lemma}
\begin{proof}
Let $(g_0, q_0)$ be an element of $\mathfrak{l}$. As the map $\pi: \Y_y\to \B_{\pi(y)}$ is a group morphism, the image of $\pi(\mathfrak{l})$ is the conjugacy class $(q_0)$ of $q_0$ in $\B_{\pi(y)}$. 

Consider the action of $\B_{\pi(y)}$ on $G$ by 
\[
\Ad_q ^\tau(g):=\tau_{\pi(y)}(q, q_0)\tau_{\pi(y)}(qq_0, q^{-1}) \tau_{\pi(y)}(q, q^{-1})^{-1}g. 
\]
By the cocycle condition of $\tau_{\pi(y)}$, it is not difficult to check that $\Ad_q^\tau$ is an action of $\B_{\pi(y)}$ on $G$.  Let $\{g_1,\cdots, g_k\}$ be the orbit of $g_0$ with respect to the action $\Ad_q^\tau$ of the group $\B_{\pi(y)}$. 

Consider an arbitrary $(g,q)$ in $\Y_y$. Then $(g,q)(g_0, q_0)(g,q)^{-1}$ can be computed as
\[
(g,q)(g_0, q_0)(g,q)^{-1}=(g,1)(1,q)(g_0, q_0)(1,q)^{-1}(g^{-1},1)=\big(g\Ad^\tau_q(g_0)g^{-1}, qq_0q^{-1}\big). 
\]
The element $g\Ad^\tau_q(g_0)g^{-1}$ belongs to the conjugacy class of one of the $g_1, \cdots, g_k$. This completes the proof of the lemma. 
\end{proof}

For  a group element $h=(g,q)\in \Y_y$, we use $h_G$ to denote the element $g\in G$.
Lemma \ref{lem:conjugacy} suggests the following definition. 
\begin{definition}\label{dfn:conjugacy} Let $\mathfrak{l}_y$ be a conjugacy class of the stabilizer group $\Y_y$ of $\Y$ at $y$. Let $\mathsf{l}_y$ be the collection of conjugacy classes $(g_i)$ of $G$ such that $(g_i)\times \pi(\mathfrak{l}_y)$ is contained inside $\mathfrak{l}_y$. 
\end{definition}

Denote $(c_1, \cdots, c_n)\in Z(G)^{\times n}$ by $\vec{c}$ and $(h_1, \cdots, h_n)\in G^{\times n}$ by $\vec{h}$. By Theorem \ref{thm:degree}, to compute the degree for the map $p_{{\LL}}$, we need to consider the following set 
\[
\chi\big(\tau, g, \vec{c}, \vec{d}, \vec{h}, (g_1), \cdots, (g_n)\big):=\left\{ 
\begin{aligned}
\big(\alpha_1, \beta_1,& \cdots,\alpha_g, \beta_g, \sigma_1, \cdots, \sigma_n\big)
\in G^{\times (2g+n)}\\
&\Big| [\alpha_1, \beta_1]\cdots [\alpha_g, \beta_g]=\prod_{i=1}^n c_i \sigma_i,\ \sigma_i\in (g_i),\\
& \sigma_1^{m_1d_1}=h_1 ^{m_1},\cdots, \sigma_n^{m_nd_n}=h_n^{m_n}
\end{aligned}
\right\}.
\]

\begin{lemma}\label{lem:union} Let $\LL=(\mathfrak{l}_1, \cdots, \mathfrak{l}_n)$ be the fixed point data of $\Y$ at $n$-points $(y_1, \cdots, y_n)$.  Let $d_i$ be the order of $\pi(\mathfrak{l}_i)$ for $i=1, \cdots, n$.  Then the degree of $p_{\LL}$ is the cardinality of the following union 
\[
\bigcup _{(g_1)\in \mathsf{l}_1, \cdots, (g_n)\in \mathsf{l}_n}
 \chi\Big(\tau, g, \big(hol(\Y, \tau)(L_i, y_i)\big), \vec{d}, \big(hol(\Y, \tau)(y_i)^{-1}\big), (g_{1}), \cdots, (g_{n})\Big)\bigg/G.
\] 
\end{lemma}
\begin{proof}
By Theorem \ref{thm:degree}, the degree of the map $p_{\LL}$ is the set of diffeomorphism classes of orbifold Riemann surfaces $\widetilde{\Sigma}$ together with the structures $p:\widetilde{\Sigma}\to \Sigma$ with the structures $p: \widetilde{\Sigma}\to \Sigma$ and $\tilde{f}:\widetilde{\Sigma}\to p^*\Y_\Sigma$, and the stabilizer group of $\Y_\Sigma$ at the $i$-th marked point is a cyclic group of order $d_i$ generated by an element in $(g_{ij_i})$ for $j_i=1, \cdots, K_i$ and all possible $d_i$. 

By Lemma \ref{lem:conjugacy}, $\mathfrak{l}_i$ is the union of $(g_{ij_i})\times \pi(\mathfrak{l}_i)$, $j_i=1, \cdots, K_i$ for $(g_{ij_i})\in \mathsf{l}_i$. And Theorem \ref{thm:degree} implies that the above set of diffeomorphism classes is isomorphic to  the following union 
\[
\bigcup _{(g_{1j_1})\in \mathsf{l}_1, \cdots, (g_{nj_n})\in \mathsf{l}_n}
 \chi\Big(\tau, g, \big(hol(\Y, \tau)(L_i, y_i)\big), \vec{d}, \big(hol(\Y, \tau)(y_i)^{-1}\big), (g_{1j_1}), \cdots, (g_{nj_n})\Big)\bigg/G.
\] 
\end{proof}

Define the following set 
\[
\chi^G_{g,c}\big((g_1), \cdots, (g_n)\big):=\left\{ 
\begin{aligned}
\big(\alpha_1, \beta_1,& \cdots,\alpha_g, \beta_g, \sigma_1, \cdots, \sigma_n\big)
\in G^{\times (2g+n)}\\
&\Big| [\alpha_1, \beta_1]\cdots [\alpha_g, \beta_g]=c\prod_{i=1}^n \sigma_i,\ \sigma_i\in (g_i)
\end{aligned}
\right\}.
\]
Similar to the above discussion, the group $G$ acts on $\Omega\big(\tau, g, \vec{c}, (g_1), \cdots, (g_n)\big)$ by conjugation. 
 
\begin{proposition}\label{prop:local-holonomy}
Let $\LL=(\mathfrak{l}_1, \cdots, \mathfrak{l}_n)$ be the fixed point data of $\Y$ at $n$-points $(y_1, \cdots, y_n)$. Then the degree of $p_{\LL}$ is the cardinality of 
\[
\bigcup _{(g_1)\in \mathsf{l}_1, \cdots, (g_n)\in \mathsf{l}_n}
\chi^G_{g, c}\Big((g_{1}), \cdots, (g_{n})\Big)\bigg/ G, 
\]
where $c=\prod_{i=1}^n hol(\Y, \tau)(L_i, y_i)$.
\end{proposition}
\begin{proof} 
By Lemma \ref{lem:union}, it is enough to prove that given any $\sigma_i\in (g_{ij_i})$, there is a unique $m_i$ such that 
\[
\sigma_i ^{m_i d_i}hol(\Y, \tau)(y_i)^{m_i}=1. 
\]
This follows from the fact that $m_id_i$ is the order of the element $(\sigma_i, \pi(\mathfrak{l}_i))$ in the finite stabilizer group $\Y_{y_i}$ of $\Y$ at $y_i$. 
\end{proof}

The cardinality of the set 
\[
\chi^G_{g,c} \big(  (g_1), \cdots, (g_n)\big) \bigg / G
\]
is computed by Eq. (\ref{eqn:formula_omega_c}) to be
\[
\begin{split}
\Omega^G_{g,c}\big(  (g_1), \cdots, (g_n)\big)&:=\frac{1}{|G|}\vert \chi^G_{g,c} \big(  (g_1), \cdots, (g_n)\big)\vert\\
&=\sum_{\alpha \in \widehat{G}}\frac{1}{\dim(V_\alpha)^n}\left(\frac{\text{dim}\, V_\alpha}{|G|}\right)^{2-2g}\alpha_c\left(\sum_{g^\bullet_1\in (g_1),...,g^\bullet_n\in (g_n)}\prod_{j=1}^n \chi_\alpha(g^\bullet_j)\right). 
\end{split}
\]
Recall that, for $\alpha\in\widehat{G}$, $\alpha_c\in \mathbb{C}$ is defined as follows: as $c\in Z(G)$, Schur's lemma implies that $\alpha(c)\in GL(V_\alpha)$ is a scalar multiple $\alpha_c Id_{V_\alpha}$. As a corollary, we have reached the following theorem. 

\begin{theorem}\label{thm:char-formula}
The local degree of the map 
\[
p_{\LL}: \mathcal{K}_{g,n}(\Y,\beta, \LL)\longrightarrow \mathcal{K}_{g,n}(\B,\beta, \pi(\LL)). 
\]
is equal to 
\begin{equation}\label{eq:degree}
d_{\vec{\mathfrak{l}}, \mathfrak{c}}=\sum_{(g_1)\in \mathsf{l}_1, \cdots, (g_n)\in \mathsf{l}_n}\sum_{\alpha \in \widehat{G}}\frac{1}{\dim(V_\alpha)^n}\left(\frac{\text{dim}\, V_\alpha}{|G|}\right)^{2-2g}\alpha_\mathfrak{c}\left(\sum_{g^\bullet_1\in (g_1),...,g^\bullet_n\in (g_n)}\prod_{j=1}^n \chi_\alpha(g^\bullet_j)\right),
\end{equation}
where $\mathfrak{c}$ is  $\prod_{i=1}^n hol(\Y_\Sigma, \tau)(L_i, z_i)$. 
\end{theorem}

\begin{remark} We observe that as $\tau$ is a local constant function, $\mathfrak{c}$ is  a $Z(G)$ valued local constant function on $\mathcal{M}_{g,n}(\B, \beta)$ and $\mathcal{K}_{g,n}(\B, \beta)$. 
\end{remark}

\subsection{Proof of Lemma \ref{lem:pushforward}}\label{subsec:vir_push} 
By Proposition \ref{prop:vir_push}, $p$ satisfies the virtual pushforward property in the sense of \cite[Definition 3.1]{m2}. Namely, if we write the unweighted virtual class $$[\Mbar_{g,n}(\B, \beta; \pi(\mathfrak{l}_1),...,\pi(\mathfrak{l}_n))]^{uw}$$ as a sum of irreducible cycles: $$[\Mbar_{g,n}(\B, \beta; \pi(\mathfrak{l}_1),...,\pi(\mathfrak{l}_n))]^{uw}=[M(\B)_1]+...+[M(\B)_s],$$
then there are rational numbers $m_1,...,m_s$ such that the pushforward of the unweighted virtual class $[\Mbar_{g,n}(\Y, \beta; \mathfrak{l}_1,...,\mathfrak{l}_n)]^{uw}$ satisfies
$$p_*[\Mbar_{g,n}(\Y, \beta; \mathfrak{l}_1,...,\mathfrak{l}_n)]^{uw}=m_1[M(\B)_1]+...+m_s[M(\B)_s].$$
The weighted virtual class satisfies
\begin{equation*}
\begin{split}
[\Mbar_{g,n}(\B, \beta; \pi(\mathfrak{l}_1),...,\pi(\mathfrak{l}_n))]^{vir}=&[\Mbar_{g,n}(\B, \beta; \pi(\mathfrak{l}_1),...,\pi(\mathfrak{l}_n))]^{uw}\cdot w(\pi(\LL))\\
=&\sum_{i=1}^s  w(\pi(\LL))[M(\B)_i].
\end{split}
\end{equation*}
Hence 
\begin{equation*}
\begin{split}
p_*[\Mbar_{g,n}(\Y, \beta; \mathfrak{l}_1,...,\mathfrak{l}_n)]^{vir}=&p_*[\Mbar_{g,n}(\Y, \beta; \mathfrak{l}_1,...,\mathfrak{l}_n)]^{uw}\cdot  w(\LL)\\
=&\sum_{i=1}^s \frac{m_i w(\LL)}{w(\pi(\LL))} \cdot w(\pi(\LL))[M(\B)_i].
\end{split}
\end{equation*}
We need to show that $$d_{\LL, i}=\frac{m_i w(\LL)}{w(\pi(\LL))}$$ for $i=1,...,s$. For each $i=1,...,s$, the number $m_i$ is computed by counting the number of liftings $\widetilde{\C}\to\Y$ of a given $[\C\to \B]\in M(\B)_i$, up to a factor $w(\pi(\LL))/w(\LL)$.  Therefore we only need to show that this counting is constant on each component of $\Mbar_{g,n}(\B, \beta; \pi(\mathfrak{l}_1),...,\pi(\mathfrak{l}_n))$. This counting for nonsingular domains $\C$ is solved in Theorem \ref{thm:char-formula}.  As $c$ takes values in $Z(G)$ and is continuous on $\Mbar_{g,b}(\B, \beta)$, it is constant on each component $\Mbar_{g,n}(\B, \beta, \pi(\mathfrak{l}_1), ..., \pi(\mathfrak{l}_n))_i$, $i=1,...,s$. To show the needed constant property for the counting number on each component $\Mbar_{g,n}(\B, \beta; \pi(\mathfrak{l}_1),...,\pi(\mathfrak{l}_n))_i$, it suffices to show that the solution to the counting problem given in Theorem \ref{thm:char-formula} is also valid for maps $\C\to \B$ with nodal domains $\C$. 

We argue as follows. First consider a map $\C\to \B$ whose domain $\C$ is a connected orbifold Riemann surface of genus $g$ with a single node $p\in \C$ obtained by identifying two points $p_1, p_2\in \C'$ on a connected nonsingular Riemann surface $\C'$ of genus $g-1$. Liftings of $\C\to \B$ are in bijection with liftings of $\C'\to \C\to \B$, where $\C'\to\C$ is the gluing map. Since $\C'$ is nonsingular, Proposition \ref{prop:local-holonomy} applies to show that the number of isomorphism classes of liftings of $\C'\to\C\to \B$ is given by 
\[
\sum_{...}\sum_{(\zeta)\in Conj(G)}|C(\zeta)|\Omega_{g-1}^G\big(..., (\zeta), (\zeta^{-1})\big).
\]
Here $...$ stands for data coming from orbifold structures at the marked points of $\C$, and the first summation $\sum_{...}$ is the corresponding summation in Theorem \ref{thm:char-formula}. $(\zeta)\in Conj(G)$ arises from liftings of orbifold structures at $p_1, p_2$. We sum over all $(\zeta)\in Conj(G)$ because we do not specify the orbifold structures at the point over $p_1$, and only require that it must map to the orbifold structure at $p_1$. The weight $|C(\zeta)|$ accounts for the possibilities of identifying the two points lying over $p_1$ and $p_2$. By the proof of \cite[Lemma 3.5]{jk}, we  have 
\[
\sum_{(\zeta)\in Conj(G)}|C(\zeta)|\Omega_{g-1}^G\big(..., (\zeta), (\zeta^{-1})\big)=\Omega_g^G(...).
\]
So the number of isomorphism classes of liftings of $\C'\to\C\to \B$ is given by $\sum_{...}\Omega_g^G(...)$, which agrees with the answer for nonsingular domains obtained in Theorem \ref{thm:char-formula}.

A similar argument proves that Theorem \ref{thm:char-formula} is also valid for maps $\C\to \B$ with domain a connected orbifold Riemann surface $\C$ with a single node $p\in \C$ obtained by gluing two nonsingular orbifold Riemann surfaces $\C_1$ and $\C_2$ along $p_1\in \C_1, p_2\in\C_2$. 

By the above two cases, and induction on the number of nodes, Theorem \ref{thm:char-formula} is proven to be valid for maps $\C\to \B$ with $\C$ any nodal Riemann surface.

\appendix

\section{Virtual pushforward property}\label{app:vir_push}
Let $p:\X_1\to \X_2$  be an \'etale morphism of smooth projective Deligne-Mumford stacks. Consider the induced map 
$$p_{g,n,\beta}: \Mbar_{g,n}(\X_1, \beta)\to \Mbar_{g,n}(\X_2, p_* \beta),$$
obtained by composing a stable map to $\X_1$ with $p$. The purpose of this Appendix is to prove the following

\begin{proposition}\label{prop:vir_push}
The map $p_{g,n,\beta}$ satisfies the virtual pushforward property in the sense of \cite[Definition 3.1]{m2}. i.e. if the {\em unweighted} virtual fundamental class $[\Mbar_{g,n}(\X_2, p_*\beta)]^{uw}$ is decomposed into a sum of irreducible cycles as $$[\Mbar_{g,n}(\X_2, p_*\beta)]^{uw}=[M(\X_2)_1]+...+[M(\X_2)_s],$$ 
then there exist $m_1,..., m_s\in \mathbb{Q}$ such that $$(p_{g,n,\beta})_*[\Mbar_{g,n}(\X_1, \beta)]^{uw}=\sum_{i=1}^sm_i[M(\X_2)_i].$$

\end{proposition}
The idea of proof is to apply \cite[Proposition 3.14]{m2}. Consider the following diagram:
\begin{equation}\label{diag:vir_class}
\xymatrix{
\Mbar_{g,n}(\X_1, \beta)\ar[ddr]_{f_1}\ar[rrd]^{p_{g,n,\beta}}\ar[dr]_{\mathfrak{q}}& & \\
 & \Mbar \ar[r]\ar[d]& \Mbar_{g,n}(\X_2,p_*\beta)\ar[d]^{f_2}\\
 & \mathfrak{M}\ar[r]^{g_2}\ar[d]_{g_1}& \mathfrak{M}_{g,n}^{tw}\\
 & \mathfrak{M}_{g,n}^{tw} &.
}
\end{equation}
We explain below the ingredients of the Diagram (\ref{diag:vir_class}). 
\begin{enumerate}

\item
$\mathfrak{M}^{tw}$ is the moduli stack of $n$-pointed genus $g$ prestable twisted curves.
\item
$f_2:\Mbar_{g,n}(\X_2, p_*\beta)\to \mathfrak{M}^{tw}$ takes a stable map $\C_2\to \X_2$ to its domain $\C_2$.

\item
$\mathfrak{M}$ is defined to be the moduli stack of  morphisms $\C_1\to \C_2$ between twisted curves such that: 
\begin{enumerate}
\item
For $i=1,2$, marked points on $\C_i$ are labelled by connected components of $I\X_i$. 
\item
The map $\C_1\to \C_2$ takes nodes to nodes, marked points to marked points, and smooth points to smooth points. 
\item
The map $\C_1\to \C_2$ is an isomorphism over the generic locus 
$$\C_1^{gen}:=\C_1\setminus (\{\text{marked points}\}\cup \{\text{nodes} \}).$$ 
\item
For an orbifold marked point $p_1\in \C_1$ which maps to $p_2\in \C_2$, the component of $I\X_1$ labeling $p_1$ must map to the component of $I\X_2$ labeling $p_2$ under the natural map $Ip: I\X_1\to I\X_2$.

\end{enumerate}
 
\item
For $i=1,2$, $g_i: \mathfrak{M}\to \mathfrak{M}^{tw}_{g,n}$ takes $\C_1\to \C_2$ to $\C_i$.

\item
$g_1\circ f_1:\Mbar_{g,n}(\X_1, \beta)\to \mathfrak{M}^{tw}$ takes a stable map $\C_1\to \X_1$ to its domain $\C_1$.

\item
The square in (\ref{diag:vir_class}) is Cartesian, in particular commutative. So the space $\Mbar$ can be identified as $\Mbar:=\mathfrak{M}\times_{\mathfrak{M}^{tw}_{g,n}} \Mbar_{g,n}(\X_2,p_*\beta)$.

\end{enumerate}
It is not hard to check that the stack $\mathfrak{M}$ is algebraic and locally of finite type. 
\begin{lemma}
The map $g_1$ is \'etale.
\end{lemma}
\begin{proof}
Since the stacks involved are algebraic and locally of finite type, it suffices to show that $g_1$ is formally \'etale, which we do by using infinitesimal lifting criteria for \'etaleness. 

Let $B$ be an artinian local ring, $I\subset B$ a square-zero ideal, and set $A:=B/I$. Let $\mathfrak{p}_A: \C_{1A}\to \C_{2A}$ be an $A$-valued object of $\mathfrak{M}$. We will show that a unique lift of $\mathfrak{p}_A$ to a $B$-valued object exists provided a lift of $\C_{1A}$ to a twisted curve $C_{1B}$ over $B$ is given. 

By \cite[Lemma 4]{Beh}, the coarse curve $C_{2A}$ of $\C_{2A}$ extends uniquely to a curve $C_{2B}$ over $B$. Moreover, the morphism $\bar{\mathfrak{p}}_A: C_{1A}\to C_{2A}$ also extends to a morphism $\bar{\mathfrak{p}}_B: C_{1B}\to C_{2B}$ over $B$.

To show that twisted curves also extend, one can use the equivalence between twisted curves and log twisted curves \cite{Ol}. The detailed argument for the existence of a lift $\mathfrak{p}_B: \C_{1B}\to \C_{2B}$ of $\mathfrak{p}_A$ is identical to the arguments in the proof of \cite[Proposition 3.4]{ajt1} and is omitted.
\end{proof}

We now discuss obstruction theory. Since $g_1$ is \'etale, the obstruction theory on $\Mbar_{g,n}(\X_1,\beta)$ relative to $g_1\circ f_1$ is the same as the obstruction theory relative to $f_1$. To describe them we consider the following diagram:

\begin{equation}
\xymatrix{
{\widetilde{\C}}\ar[dr]^{\rho} \ar[rdd]_{\widetilde{u}}\ar[rr]^{\widetilde{f}} &  & \X_1\ar[dr]^{p} & \\
 &p^*\C\ar[d]^{u'}\ar[r]^{p'_{g,n,\beta}} & \C\ar[d]^{u}\ar[r]^{f} & \X_2\\ 
 &\Mbar_{g,n}(\X_1, \beta)\ar[r]^{p_{g,n,\beta}} & \Mbar_{g,n}(\X_2, p_*\beta) & 
}
\end{equation}
where
\begin{enumerate}
\item
The map $u'$ is the pullback of $u$ via $p_{g,n,\beta}$, i.e. the square involving them is cartesian. 
\item
The map $(u, f)$ (resp. $(\widetilde{u}, \widetilde{f})$) is the universal stable map over $\Mbar_{g,n}(\X_2, p_*\beta)$ (resp. $\Mbar_{g,n}(\X_, \beta)$). 
\item
Consequently, the map $\rho$ is the universal stabilization map induced by $p: \X_1\to \X_2$. By the proof of \cite[Lemma 9.3.1]{av}, we have $\rho_*\mathcal{O}_{\widetilde{\C}}=\mathcal{O}_\C$. 
\end{enumerate}

The obstruction theory of $\Mbar_{g,n}(\X_1,\beta)$ relative to $f_1$ is given by $R\widetilde{u}_*\widetilde{f}^*T_{\X_1}\to L^\bullet_{\Mbar_{g,n}(\X_1,\beta)/\mathfrak{M}}$. The 
obstruction theory of $\Mbar_{g,n}(\X_2, p_*\beta)$ relative to $f_2$ is given by $Ru_*f^*T_{\X_2}\to L^\bullet_{\Mbar_{g,n}(\X_2, p_*\beta)/\mathfrak{M}^{tw}_{g,n}}$. We calculate
\begin{equation}\label{eqn:map_obs}
\begin{split}
Lp_{g,n,\beta}^*Ru_*f^*T_{\X_2}&\to Ru_*'L{p_{g,n,\beta}'}^*f^*T_{\X_2}\\
&\simeq Ru_*'R\rho_*L\rho^*L{p_{g,n,\beta}'}^*f^*T_{\X_2}\\
&\simeq R\widetilde{u}_*L\rho^*L{p_{g,n,\beta}'}^*f^*T_{\X_2}\\
&\simeq R\widetilde{u}_*\widetilde{f}^*p^*T_{\X_2} \\
&\simeq R\widetilde{u}_*\widetilde{f}^*T_{\X_1}.
\end{split}
\end{equation}
We explain the chain of maps in the above computation. 
\begin{enumerate}
\item
The isomorphism in the second line follows from $R\rho_*L\rho^*\simeq Id$, which follows from $\rho_*\mathcal{O}_{\widetilde{\C}}=\mathcal{O}_\C$. 
\item
The isomorphism in the third line follows from $\widetilde{u}=u'\circ \rho$. 
\item 
The isomorphism in the fourth line follows from $p\circ \widetilde{f}=f\circ p_{g,n,\beta}'\circ \rho$.
\item
The isomorphism in the fifth line follows from $T_{\X_1}=p^*T_{\X_2}$, which is because $p$ is \'etale.
\end{enumerate}
Write $\phi$ for the map obtained in Eq. (\ref{eqn:map_obs}). The following commutative diagram can then be deduced by standard arguments:
\begin{equation}
\xymatrix{
Lp_{g,n,\beta}^*Ru_*f^*T_{\X_2}\ar[r]^{\phi}\ar[d] & R\widetilde{u}_*\widetilde{f}^*T_{\X_1}\ar[d]\\
Lp_{g,n,\beta}^*L^\bullet_{\Mbar_{g,n}(\X_2, p_*\beta)/\mathfrak{M}^{tw}_{g,n}}\ar[r] & L^\bullet_{\Mbar_{g,n}(\X_1,\beta)/\mathfrak{M}}.
}
\end{equation}
Because the relative tangent bundle of $p$ is $0$, the arguments of \cite[Proposition 4.9]{m2} applies to show that  the cone of $\phi$ is a perfect complex.  

Given the above discussions, we conclude the result of Proposition \ref{prop:vir_push} from \cite[Proposition 3.14]{m2}. 
\section{A counting result}\label{app:counting_result}
Let $G$ be a finite group. Let $g, n$ be integers with $g\geq 0$ and $n>0$. Given a collection $(g_1),..., (g_n)$ of conjugacy classes of $G$, we consider the following set
\begin{equation}
\chi_g^G((g_1),...,(g_n))\subset G^{2g+n}
\end{equation}
consisting of tuples $(\alpha_1,...,\alpha_g, \beta_1,...,\beta_g, \sigma_1,...,\sigma_n)$ such that
\begin{enumerate}
\item
\begin{equation}\label{eqn:pi1_relation}
\prod_{i=1}^g [\alpha_i, \beta_i]=\prod_{j=1}^n \sigma_j, 
\end{equation}
\item
$\sigma_j\in (g_j),$  for all $j=1,...,n$,
\end{enumerate} 
Define $\Omega_g^G((g_1),...,(g_n)):=|\chi_g^G((g_1),...,(g_n))|/|G|$. In this appendix we prove the following formula:

\begin{proposition}\label{prop:formula_omega}
Suppose $\chi_g^G((g_1),...,(g_n))$ is non-empty, then
\begin{equation}\label{eqn:formula_omega}
\Omega_g^G((g_1),...,(g_n))=\frac{1}{\prod_{j=1}^n |C_G(g_j)|}\sum_{\alpha \in \widehat{G}}\left(\prod_{j=1}^n \chi_\alpha(g_j)\right) \left(\frac{\text{dim}\, V_\alpha}{|G|}\right)^{2-2g-n}.
\end{equation}
\end{proposition}

After this paper is completed, we learned from \cite{tu} that this result is due to Mednykh \cite{Me}. We have included this appendix here for completeness. 
\begin{remark}[Alternative expression]
Since the size of a conjugacy class $(g)$ satisfies $|(g)|=|G|/|C_G(g)|$, we can rewrite Eq. (\ref{eqn:formula_omega}) as
\begin{equation}\label{eqn:2nd_formula_omega}
\Omega_g^G((g_1),...,(g_n))=\sum_{\alpha \in \widehat{G}}\left(\frac{\text{dim}\, V_\alpha}{|G|}\right)^{2-2g}\left(\sum_{g^\bullet_1\in (g_1),...,g^\bullet_n\in (g_n)}\prod_{j=1}^n \chi_\alpha(g^\bullet_j)\right) .
\end{equation}
\end{remark}

\begin{remark}[Special cases]
\hfill
\begin{enumerate}
\item (Undefined cases).
For $(g_1),...,(g_n)$ such that condition (\ref{eqn:pi1_relation}) does not hold, we define the set $\chi_g^G((g_1),...,(g_n)):=\emptyset$ , and hence  $\Omega_g^G((g_1),...,(g_n))=0$ in this case. 

\item (Unstable ranges). We assume that $n\neq 0$ because we must have insertions for the purpose of comparing Gromov-Witten invariants. For $g\geq 0$ and $n>0$, the inequality $2g-2+n>0$ fails in the following two cases:
\begin{enumerate}
\item ($(g,n)=(0,1)$).
In order for (\ref{eqn:pi1_relation}) to hold, we must have $g_1=1$. Clearly $\Omega_0^G((1))=1/|G|$. The right-hand side of (\ref{eqn:formula_omega}) in this case is 
$$\frac{1}{|C_G(1)|}\sum_{\alpha_\in \widehat{G}}\chi_\alpha(1)\left(\frac{\text{dim}\, V_\alpha}{|G|}\right)^{2-1}=\frac{1}{|G|}\sum_{\alpha\in \widehat{G}}\frac{(\text{dim}\, V_\alpha)^2}{|G|}=\frac{1}{|G|}.$$
So (\ref{eqn:formula_omega}) holds in this case.

\item ($(g,n)=(0,2)$).
In order for (\ref{eqn:pi1_relation}) to hold, we must have $(g_2)=(g_1^{-1})$. Clearly we have $\Omega_0^G((g_1), (g_1^{-1}))=|(g_1)|/|G|$. The right-hand side of (\ref{eqn:formula_omega}) in this case is 
$$\frac{1}{|C_G(g_1)||C_G(g_1^{-1})|}\sum_{\alpha\in \widehat{G}}\chi_\alpha(g_1)\chi_\alpha(g_1^{-1})=\frac{|C_G(g_1)|}{|C_G(g_1)||C_G(g_1^{-1})|}=|(g_1)|/|G|$$
by orthogonality of characters. So (\ref{eqn:formula_omega}) holds in this case.

\end{enumerate}
\end{enumerate}
\end{remark}

\subsection{Proof of Proposition \ref{prop:formula_omega}}
By the above remark, we assume $2g-2+n>0$. We begin with a number of definitions, mostly taken from \cite{jk}.

For a conjugacy class $(g)$, define $e_{(g)}:=\sum_{g^\bullet \in (g)} g^\bullet$.  Clearly $e_{(g)}$ is an element of $Z(C^*(G))$, where $Z(C^*(G))$ is the center of the group algebra $C^*(G)$ of the group $G$. Under the identification $$Z(C^*(G))\simeq H_{CR}^*(BG, \mathbb{C})$$ in \cite[Corollary 3.3]{jk}, $e_{(g)}$ is identified with $1_{(g)}$, the fundamental class of the component $[\text{pt}/C_G(g)]$ of the inertia orbifold of $BG$. 

For $\alpha\in \widehat{G}$, define $\nu_\alpha=(\text{dim}\, V_\alpha/|G|)^2$ and 
\begin{equation*}
f_\alpha:=\frac{\text{dim}\,V_\alpha}{|G|}\sum_{g\in G}\chi_\alpha(g^{-1})g=\frac{\text{dim}\,V_\alpha}{|G|}\sum_{(g) \text{ conjugacy class}}\chi_\alpha(g^{-1})e_{(g)}.
\end{equation*}

For $\alpha_1,...,\alpha_n\in \widehat{G}$, define 
\begin{equation}
\nu_g(\alpha_1,...,\alpha_n):=\left(\frac{\text{dim}\, V_\alpha}{|G|}\right)^{2-2g-n}\prod_{j=1}^n\delta_{\alpha_j, \alpha}=
\begin{cases}
\left(\frac{\text{dim}\, V_\alpha}{|G|}\right)^{2-2g-n} \quad \text{ if } \alpha_1=...=\alpha_n=:\alpha\\
0 \quad \text{ otherwise}.
\end{cases}
\end{equation}

The result in \cite[Proposition 4.2]{jk}, which relates Gromov-Witten invariants of $BG$ and a point, can be stated as follows: for non-negative integers $a_1,...,a_n$
\begin{equation}\label{eqn:GW_formula}
\< \tau_{a_1}(f_{\alpha_1}),...,\tau_{a_n}(f_{\alpha_n})\>_g^G=
\begin{cases}
\nu_\alpha^{1-g}\<\tau_{a_1},...,\tau_{a_n}\>_g \quad \text{ if } \alpha_1=...=\alpha_n=:\alpha\\
0 \quad \text{ otherwise}.
\end{cases}
\end{equation}
By the definition of $f_\alpha$, we have 
\begin{eqnarray*}
&&\< \tau_{a_1}(f_{\alpha_1}),...,\tau_{a_n}(f_{\alpha_n})\>_g^G=\\
&&\qquad \qquad \qquad \left(\prod_{j=1}^n\frac{\text{dim}\, V_{\alpha_j}}{|G|}\right)\sum_{(g_1),...,(g_n) \text{ conjugacy classes}} \prod_{j=1}^n\chi_{\alpha_j}(g_j^{-1}) \< \tau_{a_1}(e_{(g_1)}),...,\tau_{a_n}(e_{(g_n)})\>_g.
\end{eqnarray*}
So we may rewrite Eq. (\ref{eqn:GW_formula}) as
\begin{equation*}
\sum_{(g_1),...,(g_n) \text{ conjugacy classes}} \prod_{j=1}^n\chi_{\alpha_j}(g_j^{-1}) \< \tau_{a_1}(e_{(g_1)}),...,\tau_{a_n}(e_{(g_n)})\>_g=\nu_g(\alpha_1,...,\alpha_n)\<\tau_{a_1},...,\tau_{a_n}\>_g.
\end{equation*}
By \cite[Proposition 3.4]{jk}, we have 
\begin{equation*}
\< \tau_{a_1}(e_{(g_1)}),...,\tau_{a_n}(e_{(g_n)})\>_g=\Omega^G_g((g_1),...,(g_n))\<\tau_{a_1},...,\tau_{a_n}\>_g.
\end{equation*}
Since we can choose $a_1,..., a_n$ such that $\<\tau_{a_1},...,\tau_{a_n}\>_g\neq 0$, we  obtain the following system of equations:
\begin{equation}\label{eqn:eqn_system}
\sum_{(g_1),...,(g_n) \text{ conjugacy classes}} \prod_{j=1}^n\chi_{\alpha_j}(g_j^{-1})\Omega^G_g((g_1),...,(g_n))=\nu_g(\alpha_1,...,\alpha_n), \quad \alpha_1,...,\alpha_n\in \widehat{G}.
\end{equation}
In what follows, we prove Eq. (\ref{eqn:formula_omega}) by solving the system introduced in (\ref{eqn:eqn_system}).

Define a matrix $\mathfrak{A}$ as follows. Columns of $\mathfrak{A}$ are indexed by $n$-tuples $((g_1),...,(g_n))$ of conjugacy classes, and rows of $\mathfrak{A}$ are indexed by $n$-tuples $(\alpha_1,...,\alpha_n)\in \widehat{G}^n$. The entry of $\mathfrak{A}$ on the column $((g_1),...,(g_n))$ and row $(\alpha_1,...,\alpha_n)$ is $\prod_{i,j=1}^n\chi_{\alpha_i}(g_j^{-1})$. $\mathfrak{A}$ is a matrix of size $C^n\times C^n$, where $C$ is the number of conjugacy classes of $G$. 

Let $\Omega$ be the column vector whose entries are indexed by $n$-tuples $((g_1),...,(g_n))$ of conjugacy classes, such that the entry at $((g_1),...,(g_n))$ is $\Omega_g^G((g_1),...,(g_n))$. Let $\nu$ be the column vector whose entries are indexed by  $n$-tuples $(\alpha_1,...,\alpha_n)\in \widehat{G}^n$, such that the entry at $(\alpha_1,...,\alpha_n)$ is $\nu_g(\alpha_1,...,\alpha_n)$. Then Eq. (\ref{eqn:eqn_system}) reads
$$\mathfrak{A}\Omega=\nu_g.$$

Define a matrix $\mathfrak{B}$ as follows. Columns of $\mathfrak{B}$ are indexed by $(\alpha_1,...,\alpha_n)\in \widehat{G}^n$, and rows of $\mathfrak{B}$ are indexed by $n$-tuples $n$-tuples $((g_1),...,(g_n))$ of conjugacy classes. The entry of $\mathfrak{B}$ on the column $(\alpha_1,...,\alpha_n)$ and row $((g_1),...,(g_n))$ is $\prod_{i,j=1}^n(\chi_{\alpha_i}(g_j)/|C_G(g_j)|)$. $\mathfrak{B}$ is also a matrix of size $C^n\times C^n$.

It follows immediately from orthogonality of characters that $\mathfrak{B}\mathfrak{A}=Id$. Hence $\Omega=\mathfrak{B}\nu_g$. Eq. (\ref{eqn:formula_omega}) follows by looking at row entries of this equality.

\subsection{A variant}
Let $c\in Z(G)$ be an element in the center. Given a collection $(g_1),..., (g_n)$ of conjugacy classes of $G$, we consider the following set
\begin{equation}
\chi_{g, c}^G((g_1),...,(g_n))\subset G^{2g+n}
\end{equation}
consisting of tuples $(\alpha_1,...,\alpha_g, \beta_1,...,\beta_g, \sigma_1,...,\sigma_n)$ such that
\begin{enumerate}
\item
\begin{equation}\label{eqn:pi1_relation_c}
\prod_{i=1}^g [\alpha_i, \beta_i]=c\prod_{j=1}^n \sigma_j, 
\end{equation}
\item
$\sigma_j\in (g_j),$  for all $j=1,...,n$,
\end{enumerate} 
Set $\Omega_{g,c}^G((g_1),...,(g_n)):=|\chi_{g,c}^G((g_1),...,(g_n))|/|G|$. The following is a variant of Proposition \ref{prop:formula_omega}.

\begin{proposition}
Suppose $\chi_{g,c}^G((g_1),...,(g_n))$ is non-empty, then
\begin{equation}\label{eqn:formula_omega_c}
\begin{split}
\Omega_{g,c}^G((g_1),...,(g_n))&=\frac{1}{\prod_{j=1}^n |C_G(g_j)|}\sum_{\alpha \in \widehat{G}}\alpha_c\left(\prod_{j=1}^n \chi_\alpha(g_j)\right) \left(\frac{\text{dim}\, V_\alpha}{|G|}\right)^{2-2g-n}\\
&=\sum_{\alpha \in \widehat{G}}\frac{1}{\dim(V_\alpha)^n}\left(\frac{\text{dim}\, V_\alpha}{|G|}\right)^{2-2g}\alpha_c\left(\sum_{g^\bullet_1\in (g_1),...,g^\bullet_n\in (g_n)}\prod_{j=1}^n \chi_\alpha(g^\bullet_j)\right)
\end{split}
\end{equation}
\end{proposition}
Here, for $\alpha\in\widehat{G}$, we define $\alpha_c\in \mathbb{C}$ as follows: Since $c\in Z(G)$, Schur's lemma imply that $\alpha(c)\in GL(V_\alpha)$ is a scalar multiple $\alpha_c Id_{V_\alpha}$. 
\begin{proof}[Proof of (\ref{eqn:formula_omega_c})]
Since $c$ is a central element, multiplication by $c$ gives a bijection between the conjugacy classes $(g)$ and $(cg)$. It follows that $\Omega_{g,c}^G((g_1),...,(g_n))=\Omega_g^G((cg_1),(g_2),...,(g_n))$. By (\ref{eqn:formula_omega}), we have 
\begin{equation*}
\Omega_{g,c}^G((g_1),...,(g_n))=\frac{1}{|C_G(cg_1)|\prod_{j=2}^n|C_G(g_j)|}\sum_{\alpha\in \widehat{G}}\chi_\alpha(cg_1)\left(\prod_{j=2}^n\chi_\alpha(g_j)\right)\left(\frac{\text{dim}\, V_\alpha}{|G|}\right)^{2-2g-n}.
\end{equation*}
Since $c\in Z(G)$, we have $|C_G(cg_1)|=|C_G(g)|$ and $\chi_\alpha(cg_1)=\alpha_c\chi_\alpha(g_1)$. The result follows.
\end{proof}

\section{Twisted GW invariants of $BG$}\label{app:twisted_bgamma}
Let $G$ be a finite group, and let $c$ be a flat $U(1)$-gerbe over $BG=[\text{pt}/G]$. The construction of \cite{pry} applies to this setting to define $c$-twisted Gromov-Witten theory of $BG$. The purpose of this appendix is to solve this theory.

The inertia orbifold of $BG$ is decomposed as $$IBG=\coprod_{(g)\subset G}BC_G(g),$$
where the disjoint union is taken over conjugacy classes $(g)$ of $G$, and $C_G(g)\subset G$ is the centralizer subgroup of $g\in G$. The flat $U(1)$-gerbe $c$ defines a line bundle $\sL_c\to IBG$, see \cite{pry} and \cite{ru1} for its construction. Classes in the cohomology $H^\bullet(IBG, \sL_c)$ of $IBG$ with coefficients in $\sL_c$ will be used in $c$-twisted Gromov-Witten theory of $BG$. 

Let $\Mbar_{g,n}(BG)$ be the moduli space of $n$-pointed genus $g$ stable maps to $BG$. The evaluation maps at marked points, $ev_i:\Mbar_{g,n}(BG)\to IBG$, take values in the inertia orbifold $IBG$. For $g_1,...,g_n\in G$, define 
$$\Mbar_{g,n}(BG; (g_1),...,(g_n)):=\cap_{i=1}^n ev_i^{-1}(BC_G(g_i)),$$
which is assumed to be non-empty. There is a natural projection $$\pi:\Mbar_{g,n}(BG; (g_1),...,(g_n))\to \Mbar_{g,n}$$ to the moduli space of $n$-pointed genus $g$ stable curves. Descendant classes on $\Mbar_{g,n}$ are denoted by $\psi_1,...,\psi_n$.

According to \cite[Section 5.2]{pry}, the line bundle $\otimes_{i=1}^n ev_i^*\sL_c$ on $\Mbar_{g,n}(BG; (g_1),...,(g_n))$ admits a trivialization $$\theta_{(g_1),...,(g_n)}:  \otimes_{i=1}^n ev_i^*\sL_c\to \underline{\mathbb{C}}.$$
For classes $\alpha_1,...,\alpha_n\in H^\bullet(IBG, \sL_c)$ and integers $a_1,...,a_n\geq 0$, the authors of \cite{pry} define  the $c$-twisted Gromov-Witten invariants to be $$\<\prod_{i=1}^n\tau_{a_i}(\alpha_i)\>_{g,n}^{BG, c}:=\int_{\Mbar_{g,n}(BG; (g_1),...,(g_n))}(\theta_{(g_1),...,(g_n)})_*(\prod_{i=1}^n ev_i^*\alpha_i)\prod_{i=1}^n(\pi^*\psi_i)^{a_i}.$$

Define $\theta(\alpha_1,...,\alpha_n):=(\theta_{(g_1),...,(g_n)})_*(\prod_{i=1}^n ev_i^*\alpha_i)\in \mathbb{C}$. The projection formula implies that 
\begin{equation}\label{eqn:proj}
\<\prod_{i=1}^n\tau_{a_i}(\alpha_i)\>_{g,n}^{BG, c}=\text{deg}(\pi) \theta(\alpha_1,...,\alpha_n)\int_{\Mbar_{g,n}}\prod_{i=1}^n \psi_i^{a_i}.
\end{equation}
The descendant integrals $\int_{\Mbar_{g,n}}\prod_{i=1}^n \psi_i^{a_i}$ are known by the Witten-Kontsevich theorem. By \cite[Proposition 3.4]{jk}, the degree $\text{deg}(\pi)$ is equal to $\Omega_g^G((g_1),..., (g_n))$, which is given by $$\Omega_g^G((g_1),...,(g_n)):=\frac{1}{|G|}\#\left\{(p_1,...,p_g, q_1,...,q_g, \sigma_1,...,\sigma_n |\prod_{i=1}^g[p_i, q_i]=\prod_{j=1}^n\sigma_j, \sigma_j\in (g_j) \right\}.$$

Genus $0$, $3$-point invariants $\<-,-,-\>_{0,3}^{BG, c}$ are used to define the Chen-Ruan cup product $\star$ on the cohomology$H^\bullet(IBG, \sL_c)$, making it a ring $QH_{orb}^\bullet(BG, c)$. It is known, see e.g. \cite[Example 6.4]{ru1}, \cite{pry}, that there is an isomorphism of rings 
\begin{equation}\label{eqn:ring_isom}
QH_{orb}^\bullet(BG, c)\simeq Z(C^*(G, c)),
\end{equation}
where $Z(C^*(G, c))$ is the center of the $c$-twisted group algebra of $G$ as is explained \cite[Examples 6.4]{ru1}.

It is known \cite[Example 6.4]{ru1} that $H^\bullet(BC_G(g), \sL_c)$ is $1$-dimensional if $(g)$ is a $c$-regular conjugacy class of $G$ (i.e. $c(g', g)=c(g, g')$ for all $g'\in C_G(g)$). If $(g)$ is not $c$-regular, then $H^\bullet(BC_G(g), \sL_c)$ vanishes. For a $c$-regular conjugacy class $(g)$, let $e_{(g)}\in H^\bullet(BC_G(g), \sL_c)$ be a generator. The collection $$\{e_{(g)} | (g) \text{ is } c\text{-regular} \}$$ additively generates $QH_{orb}^\bullet(BG, c)$. The following is \cite[Proposition 6.3]{ta-ts}: 
\begin{proposition}\label{prop:cohft}
The assignment $\Lambda_{g,n}^{BG, c}: H^\bullet(IBG, \sL_c)^{\otimes n}\to \mathbb{C}$ defined by $$\alpha_1\otimes...\otimes \alpha_n\mapsto \Lambda_{g,n}^{BG, c}(\alpha_1,...,\alpha_n):= \Omega_g^G((g_1),...,(g_n))\theta(\alpha_1,...,\alpha_n)$$
satisfies the following properties
\begin{enumerate}
\item (Forgetting tails)
$$\Lambda_{g,n}^{BG, c}(e_{(g_1)},...,e_{(g_n)})=\Lambda_{g,n+1}^{BG, c}(e_{(1)}, e_{(g_1)},..., e_{(g_n)}).$$
\item (Cutting loops)
$$\Lambda_{g,n}^{BG, c}(e_{(g_1)},...,e_{(g_n)})=\sum_{(g_0)} |C_G(g_0)| \Lambda_{g-1,n+2}^{BG, c}(e_{(g_0)}, e_{(g_0^{-1})}, e_{(g_1)},...,e_{(g_n)}).$$
\item (Cutting edge)
For $g=g_1+g_2$ and $\{1,..., n\}=P_1\coprod P_2$, 
$$\Lambda_{g,n}^{BG, c}(e_{(g_1)},...,e_{(g_n)})=\sum_{(g_0)} |C_G(g_0)| \Lambda_{g_1, |P_1|+1}^{BG, c}(\{e_{(g_i)}\}_{i\in P_1}, e_{(g_0)})\Lambda_{g_1, |P_2|+1}^{BG, c}(e_{(g_0^{-1})}, \{e_{(g_i)}\}_{i\in P_2}).$$
\end{enumerate}
Here the sums are taken over $c$-regular conjugacy classes.
\end{proposition}

Let $\widehat{G}_c$ be the set of isomorphism classes of irreducible $c$-twisted complex representations of $G$. For each $[\rho]\in \widehat{G}_c$,  a representative $V_{\rho}$ of $[\rho]$ is fixed. Similar to the untwisted case, character theory of twisted representations implies that $C^*(G, c)$ is semi-simple. More precisely, for $[\rho]\in \widehat{G}_c$, let
\begin{equation}\label{eqn:idempotent}
f_\rho:=\frac{\text{dim }V_\rho}{|G|}\sum_{g\in G^\circ}c(g, g^{-1})^{-1}\chi_\rho(g^{-1}) g.
\end{equation}
Here $G^\circ\subset G$ is the set of $c$-regular elements of $G$, i.e. those $g\in G$ such that the conjugacy class $(g)$ is $c$-regular. And $\chi_\rho$ is the character of $\rho$. By \cite[Chapter 7, Theorem 3.1]{kar}, 
\begin{equation}\label{eqn:idem_prod}
f_{\rho_1}\star f_{\rho_2}=\begin{cases}
f_\rho \quad \text{ if } \rho_1=\rho_2=:\rho\\
0 \quad \text{ else}.
\end{cases}
\end{equation}
By \cite[Chapter 7]{kar}, the orbifold Poincar\'e pairing of $\{f_\rho\}$ satisfies
\begin{equation}\label{eqn:idem_pairing}
(f_{\rho_1}, f_{\rho_2})=\begin{cases}
\nu_\rho \quad \text{ if } \rho_1=\rho_2=:\rho\\
0 \quad \text{ else}.
\end{cases}
\end{equation}
Here $\nu_\rho:=\left(\frac{\text{dim } V_\rho}{|G|}\right)^2$. 

The following result determines the $c$-twisted Gromov-Witten theory of $BG$. 
\begin{theorem}\label{thm:inv}
\begin{equation}
\<\prod_{i=1}^n\tau_{a_i}(f_{\rho_i}) \>_{g,n}^{BG, c}=\begin{cases}
\nu_\rho^{1-g} \int_{\Mbar_{g,n}}\prod_{i=1}^n \psi_i^{a_i} \quad \text{ if } \rho_1=...=\rho_n=:\rho\\
0 \quad \text{ else}.
\end{cases}
\end{equation}
\end{theorem}
\begin{proof}
By Eq. (\ref{eqn:proj}), it suffices to solve $\Lambda_{g,n}^{BG, c}(f_{\rho_1},...,f_{\rho_n})$. Its proof is identical to the one of \cite[Proposition 4.2]{jk} by induction on $g$ and $n$, with the help of the isomorphism (\ref{eqn:ring_isom}) and Proposition \ref{prop:cohft}.
\end{proof}

\end{document}